\documentclass[a4paper,12pt]{article}

\usepackage[utf8]{inputenc}
\usepackage{scrpage2}
\usepackage[backend=biber,style= numeric-comp,bibstyle=numeric,giveninits=true]{biblatex}

\renewbibmacro{in:}{%
  \ifentrytype{article}{}{\printtext{\bibstring{in}\intitlepunct}}}

\usepackage{amsmath}
\usepackage{amsfonts,amssymb,amsthm}
\usepackage{mathtools}
\usepackage{enumerate}
\usepackage{nicefrac}
\usepackage{graphicx}
\usepackage{bbm}
\usepackage{esint}
\usepackage{caption}
\usepackage{stmaryrd}
\usepackage[labelformat = brace, position=top]{subcaption}
\usepackage{mathrsfs}%für mathscr
\usepackage{textcomp}
\usepackage{gensymb} %degree symbol
 \usepackage{hyperref}

\usepackage{pgf,tikz}
\usepackage{pgfplots}
\pgfplotsset{compat=1.11}
\usetikzlibrary{calc,intersections,through,backgrounds,patterns,arrows.meta} %necessary for calculations of intersections
\usepackage{tikz-3dplot}
\usepackage[usestackEOL]{stackengine}
\usepackage{emptypage}

\newcommand\numberthis{\addtocounter{equation}{1}\tag{\theequation}} % Nummerierung
\newcounter{statement}
\newtheorem{Def}[statement]{Definition}

\newtheorem{thm}[statement]{Theorem}
\newtheorem{lemma}[statement]{Lemma}
\newtheorem{cor}[statement]{Corollary}
\newtheorem{prop}[statement]{Proposition}

\theoremstyle{remark}

\newcommand{\R}{\mathbb R} 				%real numbers

\newcommand{\ball}[2]{{B_{#2}\left(#1\right)}}				%ball of radius #2 around #1
\newcommand{\intd}{\, \mathrm{d}} 				%schönes d für Integrale
\newcommand{\warr}{\rightharpoonup}				%Pfeil für schwache Konvergenz
\newcommand{\Leb}{\mathcal L}				%Lebesgue measure
\newcommand{\Sph}{{\mathbb S}}				%Sphere in 2D
\newcommand{\Div}{\operatorname{div}}			%divergence
\newcommand{\F}{\mathcal F}								%Fourier Transform
\newcommand{\supn}{^{(n)}}
\newcommand{\supdelta}{^{(\delta)}}
\newcommand{\RE}{\operatorname{Re}}
\newcommand{\supp}{\operatorname{supp}}
\newcommand{\dist}{\operatorname{dist}}

\newcommand{\tr}{\operatorname{tr}}
\newcommand{\eps}{\varepsilon}
\newcommand{\interior}[1]{{\kern0pt#1}^{\mathrm{o}}	%Definition of interior set
}
\newcommand{\translatepoint}[1]%			For translation of coordinate system in 3D tikzpictures
{   \coordinate (middlepoint) at (#1);
}

\def\Xint#1{\mathchoice
{\XXint\displaystyle\textstyle{#1}}%
{\XXint\textstyle\scriptstyle{#1}}%
{\XXint\scriptstyle\scriptscriptstyle{#1}}%
{\XXint\scriptscriptstyle\scriptscriptstyle{#1}}%
\!\int}
\def\XXint#1#2#3{{\setbox0=\hbox{$#1{#2#3}{\int}$ }
\vcenter{\hbox{$#2#3$ }}\kern-.58\wd0}}

\def\dashint{\Xint-}

\title{Quantitative aspects of the rigidity of branching microstructures in shape memory alloys via H-measures}
\author{Thilo M. Simon\footnote{Research carried out at: Max-Planck-Institut f\"ur Mathematik in den Naturwissenschaften, Inselstra{\ss}e 22, 04103 Leipzig, Germany. Now at:  New Jersey Institute of Technology, University Heights Newark, New Jersey 07102, USA.  Please use {tmsimon@njit.edu} for correspondence. ORCID iD: 0000-0003-4323-7692.}}

\bibliography{bibliography.bib}
\begin{document}
\maketitle
\begin{abstract}
	We quantify the rigidity of branching microstructures in shape memory alloys undergoing cubic-to-tetragonal transformations in the geometrically linearized theory by making use of Tartar's H-measures.
	The main result is a $B^{2/3}_{1,\infty}$-estimate for the characteristic functions of twins, which heuristically suggests that the larger-scale interfaces can cluster on a set of Hausdorff-dimension $3-\frac{2}{3}$.
	We provide evidence indicating that the dimension is optimal.
	Furthermore, we get an essentially local lower bound for the blow-up behavior of the limiting energy density close to a habit plane.
    \medskip


    \medskip

\end{abstract}

\section{Introduction}
\subsection{Literature}
The use of tools to measure the failure of strong compactness in the analysis of microstructure has a long tradition, with Young measures being the most prevalent choice.
An overview of their application in this context can be found in notes by M\"uller \cite{muller_microstructure}.
However, while Yound measures are capable of detecting the oscillations of a fine twin, they are insensitive to their geometry.
Therefore, they are the wrong tool to use as in the present work we want to capture the rigidity due to the microscopic geometry of branching microstructures after having described their macroscopic geometry in the paper \cite{simon2017rigidity}.
Instead, we make use of Tartar's H-measures \cite{Tartar90}, independently defined by Gerard \cite{gerard1988microlocal}, as they are well-suited to detect the essentially one-dimensional oscillations of small-scale twinning.
What is more, their transport property \cite[Section 3]{Tartar90}, which describes how a linear PDE for the sequence restricts the transport of oscillations, make them a natural tool to analyze rigidity properties.
%A nice feature of this approach is the tractability and conceptual clarity of the computations, which allows us to at least partially include austenite in this chapter in contrast to the preceding ones.

% Further examples can be found in works of Chipot and Kinderlehrer \cite{chipot1988equilibrium}, Kinderlehrer
% 
% CHIPOT und KINDERLEHRER, KINDERLEHRER UND ANDERE
% 
% Ball and James \cite{BallJames92Youngmeasure}.

Instead of a non-linear approach in the spirit of Ball and James \cite{BallJames87minimizers} we choose the geometrically linearized theory for the basis of our analysis.
It was first used by Khatchaturyan, Roitburd and Shatalov \cite{khachaturyan1967some, khachaturyan1983theory, khachaturyan1969theory, roitburd1969domain, roitburd1978martensitic} to model materials undergoing martensitic phase transformations. 
Somewhat more recently it has been used to provide rigorous rigidity results and constructions of microstructures:

Dolzmann and M\"uller \cite{DM95} proved that twins are the only stress-free microstructures in cubic-to-tetragonal transformations.
Capella and Otto \cite{CO09,CO12} quantified their result by augmenting the elastic energy with an interface penalization.
A simplified scalar version of such a functional has previously been used in the well-known works by Kohn and M\"uller \cite{KM92,KM94branching} to argue that the interface energy leads to a branching of twins at a habit plane.
The microscopic structure of minimizers for their model was investigated by Conti \cite{Conti00branched}, establishing their asymptotic self-similarity.
An analysis of the large-scale structure of microstructures locally involving at most two martensite phases has been given by the author \cite{simon2017rigidity}.

Previous applications of H-measures in the theory of shape memory alloys have been given by Kohn \cite{Kohn1991relaxation}, who used H-measures to calculate the quasiconvex envelope of a two-well energy in the geometrically linear theory, and Smyshlyaev and Willis \cite{smyshlyaev1999relation} and Govindjee, Hall and Mielke \cite{GOVINDJEE2003I}, who analyzed the three-well and the n-well case, respectively, building on Kohn's work.
Additionally, H-measures have been used by Heinz and Mielke \cite{heinz2016existence} to study the existence of solutions to a rate-independent model for dynamics in a two-well phase transformation.

\subsubsection*{Outline}
In Subsection \ref{subsec:def_energy} we give the energy and its elementary properties.
The main results are collected in Subsection \ref{subsec:main_results}.
Section \ref{sec:statements} contains a discussion of the necessary intermediate statements, while the proofs are given in Section \ref{sec:proofs}.

\subsection{Definition of the energy}\label{subsec:def_energy}
In the following, we give a definition of the energy and repeat the properties of the energy of direct relevance to our problem in order to fix notation.
For a more thorough discussion of the model see Capella and Otto \cite{CO12}.

As in the companion paper \cite{simon2017rigidity}, we only consider sequences $(u_\eta, \chi_\eta)$ with \[\limsup_{\eta \to 0} E_\eta(u_\eta,\chi_\eta)<\infty,\]
where the energy is given by
\begin{align}
  E_\eta(u,\chi)&:= E_{elast}(u,\chi)+ E_{inter,\eta}(u,\chi),\numberthis \label{ENERGIE!}\\
  \intertext{for}
  E_{elast,\eta}(u,\chi)&:=\eta^{-\frac{2}{3}}\int_\Omega \left| e(u) -\sum_{i=1}^3\chi_i e_i\right|^2 \intd \Leb^3,\\
  E_{inter,\eta}(u,\chi)&:=\eta^{\frac{1}{3}} \sum_{i=1}^3|D \chi_i|(\Omega).
\end{align}
Here the set $\Omega\subset \R^3$ is a bounded Lipschitz domain, the displacement is a function $u:\Omega \to \mathbb{R}^3$ and the strain is denoted by $e(u) =\frac{1}{2}\left(Du + Du^T\right)$.
Additionally, the maps $\chi_i:\Omega \to \{1,1\}$ for $i=1,\ldots,3$ with $ \sum_{i=1}^3\chi_i = 1$ represent the partition into the phases, and the martensite strains are
\[
    e_0 := 0,
    e_1:=\begin{pmatrix}
             -2 & 0 & 0 \\
             0 & 1 & 0\\
             0 & 0 & 1\\
	  \end{pmatrix},
    e_2:=\begin{pmatrix}
             1 & 0 & 0 \\
             0 & -2 & 0\\
             0 & 0 & 1\\
	\end{pmatrix},         
    e_3:=\begin{pmatrix}
             1 & 0 & 0 \\
             0 & 1 & 0\\
             0 & 0 & -2\\
	\end{pmatrix}.\numberthis \label{strains}
\]
In contrast to the previous paper \cite{simon2017rigidity}, note that we are considering the greater generality of austenite being present.

The martensite strains are (symmetrically) rank-one connected via
\begin{align}\label{rank_one_connections}
	\begin{split}
		e_2 - e_1 & = 6 \, \nu_3^+ \odot \nu_3^- = 6 \cdot \frac{1}{2}\left( \nu_3^+ \otimes \nu_3^- + \nu_3^- \otimes \nu_3^+ \right),	\\
		e_3 - e_2 & = 6\, \nu_1^+ \odot \nu_1^-,\\
		e_1 - e_3 & = 6\, \nu_2^+ \odot \nu_2^-.
	\end{split}
\end{align}
Here, the normals are defined as
\begin{align}
  \label{def:normals}
 \begin{split}
	\nu_1^+:= \frac{1}{\sqrt 2}(011), \nu_1^-:=\frac{1}{\sqrt 2}(01\overline1),\\
	\nu_2^+:= \frac{1}{\sqrt 2}(101), \nu_2^-:=\frac{1}{\sqrt 2}(\overline101),\\
	\nu_3^+:= \frac{1}{\sqrt 2}(110), \nu_3^-:=\frac{1}{\sqrt 2}(1\overline10).
 \end{split}
\end{align}
Note that we employ crystallographic notation in defining $\overline1 := -1$.
We collected the normals in the three pairs 
\begin{align*}
	N_1:= \{\nu_1^+,\nu_1^-\},\\
	N_2:=\{\nu_2^+,\nu_2^-\},\\
	N_3:= \{\nu_3^+,\nu_3^-\}.
\end{align*}
and denote their union by $N:=N_1\cup N_2 \cup N_3$.

%\begin{remark}\label{rem:combinatorics}
%An essential combinatorial property is that for any $\nu_i \in N_i$, $\nu_{i+1} \in N_{i+1}$ with $i\in \{1,2,3\}$ there exists exactly one $\nu_{i-1} \in N_{i-1}$ such that $\{\nu_{i},\nu_{i+1},\nu_{i-1}\}$ is linearly dependent:
%Indeed, the linear relation is given by $\nu_j \cdot d = 0$ for a space diagonal
%\[d\in \mathcal{D}:= \{[111],[\overline111],[1\overline11],[11\overline1]\}\numberthis\label{intro:space_d}\]
%of the unit cube, see Figure \ref{fig:geometry_normals_b}.
%We will prove in Step 1 of the Proof of Proposition \ref{prop: six_corner} that they form $120\degree$ angles.
%Additionally, for every $\nu \in N$ there exist precisely two $d\in \mathcal{D}$ such that $\nu\cdot d = 0$ and for $\nu \in N_i$ and $\tilde \nu \in N_{i+1}$ there exists a single $d \in \mathcal{D}$ such that $\nu \cdot d = \tilde \nu \cdot d =0$.
%In contrast, for each $d\in \mathcal{D}$ we have $\nu_i^+\cdot d=0$ and $\nu_i^- \cdot d \neq 0$ or vice versa.
%\end{remark}

In order to localize our results we will at times think of $E_{\eta}$, $E_{elast,\eta}$ and $E_{inter,\eta}$ as finite Radon measures on $\Omega$, where we dropped the dependence on $u_\eta$ and $\chi_\eta$.
Furthermore, passing to a subsequence we assume the existence of finite Radon measures $E_{elast}$ and $E_{inter}$ on $\Omega$ such that $E_{elast,\eta} \overset{\ast}{\warr} E_{elast}$ and $E_{inter,\eta} \overset{\ast}{\warr} E_{inter}$ as measures.
For the Lebesgue-densities of the limiting energy we will use the abbreviations
\begin{align}\label{Lebesgue_densities_notation}
	\begin{split}
		E_{elast}^\Leb  & := \frac{DE_{elast}}{D\Leb^3}\\
		E_{inter}^\Leb & := \frac{DE_{inter}}{D\Leb^3}\\
		E_{elast}^\Leb(U) &  := \int_U E_{elast}^\Leb \intd \Leb^3,\\
		E_{inter}^\Leb(U) & := \int_U E_{inter}^\Leb \intd \Leb^3
	\end{split}
\end{align}
for $U \subset \Omega$.
Additionally, let $U_h := U + \ball{0}{h}$.

Finally, observe that the weak$^*$ limits $\theta_i$ of the functions $\chi_i$ relate to the limiting displacement via
\begin{equation}\label{relation_Du_theta}
  \partial_iu_i  = -3\theta_i - \theta_0 + 1
\end{equation}
for $i=1,2,3$.
This is a straightforward consequence of the computation
\begin{equation}\label{relation_Du_chi}
 \begin{split}
  \partial_iu_{i,\eta} & = \sum_{j=0}^3 \chi_{j,\eta}(e_j)_{ii} + o_{L^2}(\eta)= -2\chi_{i,\eta} + \sum_{j=1, j\neq i}^3 \chi_{j,\eta}  + o_{L^2}(\eta)\\
  & = -3\chi_{i,\eta} - \chi_{0,\eta}+ 1 + o_{L^2}(\eta),
 \end{split}
\end{equation}
where we used $\sum_{i=0}^3 \chi_{i,\eta} = 1$.
%Note that we will prove $\theta_0 \in \{0,1\}$ in Proposition \ref{prop:disjoint_supports}, so that equation \eqref{relation_Du_theta} boils down to
%\[\partial_i u_i = (-3\theta_i +1)\chi_{\{\theta_0 = 0\}}.\]

The martensite indices $1,2$ and $3$ will be used cyclically. 
Note that the austenite index $0$ is explicitly excluded from this convention.

\subsection{Main results}\label{subsec:main_results}
Our two main contributions state that, as long as the volume fractions of any mixture of martensites does not degenerate towards a pure phase, the characteristic functions of the twins in a finite energy sequence belong to the space $B^{2/3}_{1,\infty}$.
In view of Definition \ref{def:Besov}, this roughly says that they have two-thirds of a derivative in $L^1$, or rather, that the fractional derivative is a measure.
In particular, we do not get that interfaces between twins form a 2-rectifiable set.
Instead, the estimate corresponds to the set of interfaces having at most Hausdorff-dimension $3 - \frac{2}{3}$.

Furthermore, there is plenty of evidence that this dimension is sharp:
First, in Proposition \ref{prop:disjoint_supports} we prove using the rescaling properties of the functional that the set on which the Capella-Otto result \cite{CO12} cannot be applied after blow-up is of at most the same dimension.
Secondly, it is straightforward to construct second-order laminates with finite energy such that the large-scale interfaces cluster on sets of Hausdorff-dimensions $3-\frac{2}{3} - \eps$ for all $\eps >0$.
Also this is mostly a result of scaling: The energy between two large-scale interfaces can easily be seen to scale as $d^{2/3}$, where $d$ is the distance between the interfaces.

Theorem \ref{thm:Besov_for_non-checkerboards} deals with the case that in the limit there is at least some amount of twinning everywhere, i.e., that the volume fractions are bounded away from pure phases.
This takes care of most two-variant configurations, second-order laminates and triple intersections in the terminology of \cite[Definitions 2.4, 2.6, 2.8]{simon2017rigidity}.
However, it excludes the presence of austenite.

\begin{thm}\label{thm:Besov_for_non-checkerboards}
	There exist universal constants $c,C \geq 1$ with the following property:
	
	Let $(u,\theta)$ be the limit of a finite energy sequence of displacements and partitions.
% with H-measures $\mu_i\supdelta$.
	Furthermore, assume that $\theta_i < 1$ for $i=0,\ldots 3$ almost everywhere on $\Omega$ and let there exist $\eps \geq 0$ such that for all $i=1,2,3$ we have $18 \theta_i (1-\theta_i) \geq \eps$ on the set $\{0 < \theta_i < 1\}$.
	
	Then the characteristic function of the twin normal to $\nu \in N$ in the sense of the decomposition of Lemma \ref{lem:structmeas3D} and Corollary \ref{cor:diff_incl} satisfies $\chi_{[\nu]} \in B^{2/3}_{1,\infty}(\Omega)$ with the estimate
	\[\int_U |\partial_d^h \chi_{[\nu]}(x) |\intd x \leq C \eps^{-1}\left(E_{inter}^\Leb(U_{ch})\right)^\frac{2}{3}\left(E_{elast}^\Leb(U_{ch})\right)^\frac{1}{3} h^\frac{2}{3}\]
	for all $d\in \Sph^1$, open sets $U \subset \subset \Omega$ and $ h < \frac{1}{c}\dist(U,\partial \Omega)$.
	For definitions of $E_{inter}^\Leb$ and $E_{elast}^\Leb$ see equations \eqref{Lebesgue_densities_notation}.
\end{thm}

There is a corresponding version of this statement, Theorem \ref{thm:Besov_checkerboard}, for planar checkerboards, which do exhibit pure phases.

\begin{thm}\label{thm:Besov_checkerboard}
	There exist universal constants $c,C\geq 1$ with the following property:
	
	Let $(u,\theta)$ be the limit of a finite energy sequence of displacements and partitions.
	% with H-measures $\mu_i\supdelta$.
	
	Assume that $e(u)$ is a planar checkerboard in the sense of \cite[Definition 2.7]{simon2017rigidity}:
	There exists $i\in \{1,2,3\}$ such that
	\begin{align}\label{checkerboard_h_measure}
		\begin{split}
 				 \theta_i(x) = &  - a \chi_A (x\cdot \nu_{i+1})  -  b \chi_B(x\cdot \nu_{i-1})  + 1,\\
 				 \theta_{i+1}(x) = &  \phantom{{}-{}a \chi_A(x\cdot \nu_{i+1})  {}+{} } b \chi_B (x\cdot \nu_{i-1}) ,\\
				 \theta_{i-1}(x) = & \phantom{ {}-{}} a \chi_A(x\cdot \nu_{i+1}) 
		\end{split}
	\end{align}
	with $\nu_{j} \in N_j$ for $j \in \{1,2,3\}\setminus\{i\}$, measurable sets $A,B \subset \R$ and real numbers $a,b \geq 0 $ such that $a + b = 1$.
	Let us furthermore suppose that $a>0$ and $b>0$.
	
	Then we have that $\chi_{[\nu]} \in B^{2/3}_{1,\infty}(\Omega)$ for all $\nu \in N$ with the estimate
	\[\int_U |\partial_d^h \chi_{[\nu]}(x) |\intd x \leq  \frac{C}{\min(a,b)}\left(E_{inter}^\Leb(U_{ch})\right)^\frac{2}{3}\left(E_{elast}^\Leb(U_{ch})\right)^\frac{1}{3} h^\frac{2}{3}\]
	for all $d\in \Sph^1$, open sets $U\subset\subset \Omega$ and $h < \frac{1}{c}\dist(U,\partial\Omega)$.
	Furthermore, we have the same estimate for the characteristic functions \[\chi_{\{\theta_1= 0,\, \theta_2 = b,\, \theta_3 = a\}}\text{, }\chi_{\{\theta_1= 1-b,\, \theta_2 = b,\, \theta_3 = 0\}}\text{, }\chi_{\{\theta_1= 1-a,\, \theta_2 = 0,\, \theta_3 = a\}} \text{ and } \chi_{\{\theta_1= 1,\, \theta_2 = 0,\, \theta_3 = 0\}}\]
	of the sets on which $\theta$ is constant.
\end{thm}

For the convenience of the reader, we give the definition of the relevant Besov space $B^{2/3}_{1,\infty}$.

\begin{Def}[{\cite[Chapter 1.10.3]{Triebel2}}]\label{def:Besov}
	For a function $f :\Omega \to \R$ let
	\[
		\partial_d^h f(x,\Omega) := \begin{cases}
			f(x + hd) - f(x) & \text{ if } x, x+ hd \in \Omega,\\
			0 & \text{ otherwise.}
		\end{cases}
	\]
	The Besov space $B^{2/3}_{1,\infty}(\Omega)$ can be defined as
	\[ B^{2/3}_{1,\infty}(\Omega) := \left\{f \in L^1(\Omega): \sup_{0<h\leq 1, d \in \Sph^2} |h|^{-\frac{2}{3}} ||\partial_d^h f(\bullet, \Omega)||_{L_1(\Omega)}<\infty \right\}.\]
\end{Def}
Note that we will drop the dependence of the difference operator $\partial_d^h$ on the domain whenever it is clear that $x, x+hd \in \Omega$.

Finally, with the methods developed in the paper we can also straightforwardly prove an essentially local lower bound on how the limiting energy concentrates close to a macroscopic interface.
It states that the energy density in a twinned region has to blow-up as $\tilde d^{-2/3}$, where $\tilde d$ is the distance to the interface.
We also expect this estimate to be optimal as it nicely fits the scaling $d^{1/3}$ for the energy between two macroscopic interfaces of distance $d$.
Furthermore, it is the expected scaling for (approximately) self-similar minimizers of the the Kohn-M\"uller functional, see Conti \cite{Conti00branched}.

For reasons of brevity we only state the lemma in the case of a habit plane.
However, a similar estimate is true on the both sides of an interface between two martensite twins with essentially the same proof.

\begin{lemma}\label{energy_density_habit_plane}
	There exists a universal constant $C>0$ with the following property:
	
	Let $\nu_1 \in N_1$ and let $\Omega =  \{x' \in \ball{0}{1}: x\cdot \nu_1 = 0\} + (-1,1) \nu_1$.
	Let $(u,\theta)$ be the limit of a finite energy sequence of displacements and partitions.
%	with H-measures $\mu_i\supdelta$ for $i=1,2,3$ and $\delta \geq 0$ on $\Omega$.	
	Furthermore, let the volume fractions $\theta$ and the H-measures describe a habit plane at $x\cdot \nu_1 = 0$ joining austenite with the variants $e_1$ and $e_2$ twinned in direction $\nu_3$, see also Figure \ref{fig:habit_plane_sketch}:
	\begin{enumerate}
		\item We have
			\begin{align*}
				\theta_0 & \equiv \chi_{(-1,0)}(\bullet \cdot \nu_1),\\
				\theta_1 & \equiv \frac{1}{3}\chi_{(0,1)}(\bullet \cdot \nu_1),\\
				\theta_2 & \equiv \frac{2}{3}\chi_{(0,1)}(\bullet \cdot \nu_1),\\
				\theta_3 & \equiv 0,
			\end{align*}
			which is equivalent to
			\[e(u) \equiv \chi_{(-1,0)}(\bullet\cdot \nu_1)e_0 + \chi_{(0,1)}(\bullet\cdot \nu_1) \left(\frac{1}{3}e_1 + \frac{2}{3}e_2\right).\]
		\item There exist $\nu_3 \in N_3$ such that
	\[\chi_{[\nu_3]} \equiv \chi_{(0,1)}(\bullet\cdot \nu_1).\]
	\end{enumerate}
	
	Then for any direction $d \in \Sph^2$ transversal to the habit plane and normal to the direction of twining, i.e., such that $d\cdot \nu_1 >0$ and $d\cdot \nu_3 = 0$, the following holds:
	For any $0<h$ small enough and $\mathcal{H}^2$-almost all $x' \in \R^3$ with $x' \cdot \nu_1 = 0$ and $|x'|<1$ the energy densities satisfy the lower bound
	\[ \left(E_{inter}^\Leb\right)^\frac{2}{3}(x' + h d) \left(\dashint_0^h E_{elast}^\Leb (x' + sd) \intd s \right) ^\frac{1}{3} \geq  C h^{-\frac{2}{3}}.\]
\end{lemma}

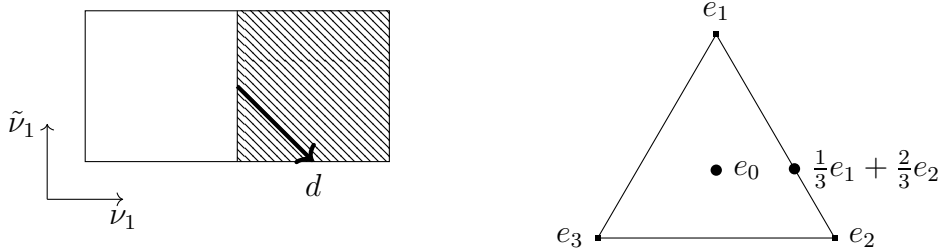
\begin{figure}[htbp]
 \begin{subfigure}{.45\linewidth}
  \centering
  \begin{tikzpicture}
   \draw (-2,-1) -- (2,-1) -- (2,1) -- (-2,1) -- cycle;
   \draw (0,-1) -- (0,1);
   \fill[pattern=north west lines] (0,-1) -- (0,1) -- (2,1) -- (2,-1);
   \draw[->] (-2.5,-1.5) -- (-1.5,-1.5) node[anchor=north] {$\nu_1$};
   \draw[->] (-2.5,-1.5) -- (-2.5,-.5) node[anchor=east] {$\tilde \nu_1$};
   \draw[->,ultra thick] (0,0) -- (1,-1) node[anchor=north] {$d$};
  \end{tikzpicture}
 \end{subfigure}
 \begin{subfigure}{.45\linewidth}
  \centering
 	\begin{tikzpicture}[scale=1.8]
% 	  \node at (-1,1) {b)};
	  \draw (90:1) -- (210:1) -- (330:1) -- cycle;
	  \node [fill=black,inner sep=1pt,label=90:$e_1$] at (90:1) {};
	  \node [fill=black,inner sep=1pt,label=180:$e_3$] at (210:1) {};
	  \node [fill=black,inner sep=1pt,label=0:$e_2$] at (330:1) {};
	  \node [circle,fill=black,inner sep=1.5pt,label=0:$e_0$] at (0,0) {};
	  \node [circle,fill=black,inner sep=1.5pt,label=0:$\frac{1}{3}e_1 + \frac{2}{3}e_2$] at ($(90:1)!.66!(330:1)$) {};
% 	  \node at (30:1) {$V_2$};
% 	  \node at (150:1) {$V_3$};
% 	  \node at (270:1) {$V_1$};
%  	  \node at (0:0.3) {$\K$};
	\end{tikzpicture} 
 \end{subfigure}
 \caption{
  The sketch on the left shows the support of the H-measures $\sigma_1=\sigma_2$ on the cross-section $\Omega \cap \{x_1 = 0\}$ with $\tilde \nu_1 \in N_1\setminus\{\nu_1\}$.
  The blank area corresponds to austenite, and the hatched area indicates twinning with normal $\nu_3$.
  The plot on the right-hand side shows the average strains.
 }
 \label{fig:habit_plane_sketch}
\end{figure}

Finally, it is interesting to note that all estimates only depend on the density of the limiting energy measure with respect to Lebesgue measure.
This is consistent with the energy contribution on boundary layers close to the habit plane being of lower order in constructions of habit planes, see for example \cite{KM94branching, CO09}.

\section{Intermediate statements}\label{sec:statements}
\subsection{Existence of the H-measures}\label{subsec:H-measure_existence}
A straightforward application of Korn's inequality ensures that after subtraction of a skew-symmetric linear function $Du_\eta$ is bounded in $L^ 2(\Omega)$.
Thus, after subtracting constants and passing to a subsequence we get the existence of $u\in W^{1,2}(\Omega, \R^3)$ such that $u_\eta \rightharpoonup u$ in $W^{1,2}(\Omega, \R^3)$.

As we will see later, we have to regularize the displacement for the transport property to hold.
To this end, we consider $u^{(\delta)}_{\eta}:= \varphi_{\delta\eta^{\frac{1}{3}}}\ast u_\eta$ for $\delta>0$, where $\varphi$ is a smooth, radially symmetric convolution kernel supported on $B_1(0)$.
A weak-times-strong argument proves that for each $\delta>0$ we still have $u^{(\delta)}_{\eta} \rightharpoonup u$ in $W^{1,2}_{loc}$.

By the existence theorem for H-measures \cite[Theorem 1.1]{Tartar90} we can extract a subsequence such that for $i,j,k =1,2,3$ the H-measures $\mu_i\left(E_j,E_k; \psi_1\psi_2^* \otimes a \right)$ of the pairs of sequences $\partial_j u_{i,\eta}-\partial_j u_i$ and $\partial_k u_{i,\eta}-\partial_k u_i$ exist as limits of
\[\int_{\R^ 2}\F\left(\psi_1\left(\partial_j u_{i,\eta}-\partial_j u_i\right)\right)\F^*\left(\psi_2\left(\partial_k u_{i,\eta}-\partial_k u_i\right)\right)a\left(\frac{\xi}{|\xi|}\right)\intd \xi\]
for $\psi_1,\psi_2:\R^3 \to \mathbb{C}$ and $a: \Sph^2 \to \mathbb{C}$.
By linearity, it is sufficient to consider the case that $\psi_1,\psi_2$ and $a$ take values in $\R$.
Note that we follow Tartar \cite{Tartar90} in using the convention
\begin{align*}
	\F f(\xi) = \int_{\R^3} f(x) e^{-2\pi i x\cdot \xi}\intd x.
\end{align*}
Furthermore, we may assume that the H-measures $\mu_i\supdelta\left(E_j,E_k; \bullet \otimes \bullet \right)$ associated to the sequences $\partial_j u^{(\delta)}_{i,\eta}-\partial_j u_i$ and $\partial_k u^{(\delta)}_{i,\eta}-\partial_k u_i$ for $i,j,k=1,2,3$ exist along a subsequence for a countable, dense subset of $\{\delta>0\}$.
The following straightforward lemma ensures that the convergence in fact extends to all $\delta>0$.

\begin{lemma}\label{lemma:H_measure_non-dense_set}
 If for $i,j,k=1,2,3$ the H-measures $\mu_i\supdelta\left(E_j,E_k; \bullet \otimes \bullet \right)$ exist for all parameters $\delta \in N \subset (0,\infty)$ with $\bar N = [0,\infty)$, then they also exist for $\delta \in \bar N \cap (0,\infty)$.
\end{lemma}

For convenience, we will set $\mu_k^{(0)} := \mu_k$.
As H-measures are bilinear in their generating sequences, the H-measures $\mu_k\supdelta\left(v,w;\bullet\otimes \bullet\right)$ for $k=1,2,3$ and $\delta \geq 0$ associated to the partial derivatives in all directions $v,w \in \R^3$ exist.
In fact, we can think of $\mu_k\supdelta$ as measure-valued bilinear forms on $\R^3$.
% for all $\psi \in C_c(\R^3,\mathbb{C})$ and $a \in C(\Sph^2, \mathbb{C})$.

% AUSKOMMENTIEREN!
% 
% (Of course we could also pair the derivatives of $u_{\eta,i}-u_i$ with $u_{\eta,j}-u_j$ for $i\neq j$, but the corresponding H-measures should not carry any additional information.
% Using the Cauchy-Schwartz inequality one can see that it can only have support where the H-measures of $u_{\eta,i}-u_i$ and $u_{\eta,j}-u_j$, paired with themselves, have support.
% Using the more detailed information derived later and the rigidity result in \cite{CO12} it should be possible to identify it as the common components of the latter measures.)

\subsection{Structure of the H-measures}\label{subsec:H-measure_structure}

We begin analyzing the H-measures by noting that the displacements solve six inhomogeneous wave equations, which result from an interplay between the general integrability condition $\partial_i g_j = \partial_j g_i$ of a gradient field $g$ and the symmetric gradient almost being diagonal and trace-free in our problem.
As we will later want to have fully localized statements, we make sure that the local dependence of the inhomogeneities on the energy density is reflected in the statement.

\begin{lemma}\label{lemma:H-measure_wave_equations}
 There exists a universal constant $c>0$ with the following property:
 The displacements $u\supdelta_\eta$ satisfy the differential constraints
 \begin{align}
  \begin{split}
   \partial_{[111]}\partial_{[\overline 111]} u\supdelta_{1,\eta}& = \Div h\supdelta_{1,\eta},\\
   \partial_{[1\overline11]}\partial_{[ 11\overline1]} u\supdelta_{1,\eta} & = \Div h\supdelta_{2,\eta},\\
   \partial_{[1\overline11]}\partial_{[111]} u\supdelta_{2,\eta} & = \Div h\supdelta_{3,\eta},\\
   \partial_{[\overline111]}\partial_{[ 11\overline1]} u\supdelta_{2,\eta} & = \Div h\supdelta_{4,\eta},\\
   \partial_{[111]}\partial_{[ 11\overline1]} u\supdelta_{3,\eta} & = \Div h\supdelta_{5,\eta},\\
   \partial_{[1\overline11]}\partial_{[\overline 111]} u\supdelta_{3,\eta} & = \Div h\supdelta_{6,\eta}.
  \end{split}
 \end{align}
 Here the vector fields $h_{i,\eta}\supdelta :\Omega \to \R^3$ for $\delta>0$ and $i=1,\ldots, 6$  satisfy
 \[h_{i,\eta}\supdelta = \varphi_{\delta \eta^\frac{1}{3}} \ast h_{i,\eta}^{(0)}\]
 on $\{x \in \Omega: \dist(x,\partial \Omega) > \delta\eta^\frac{1}{3}\}$.
 Furthermore, we have
 \[\int_\Omega \psi^2 |h^{(0)}_{i,\eta}|^2  \intd x \leq c \eta^{\frac{2}{3}} E_{elast,\eta}(\psi^2)\numberthis\label{3D:error_constraint}\]
 for all $\psi \in C_c(\Omega;\R)$.
 In particular, we have $h_{i,\eta}\supdelta \to 0$ in $L^2_{loc}(\Omega)$ for $i=1,\ldots, 6$ and $\delta \geq 0$.
%  \[\int_\Omega \psi^2 |h\supdelta_{\eta}|^2  \intd x \leq c \eta^{\frac{2}{3}} E_{elast,\eta}(u_\eta,\chi_\eta)\]
%  for all $\psi \in C_c(\Omega;[0,1])$ such that $\supp(\psi) + \ball{0}{\delta} \subset \Omega$.
\end{lemma}

The localization principle of H-measures states that linear differential constraints such as those given above contain information about the support of the measures in the Fourier variable.
In general, it allows to use the fact that the H-measures are generated by gradients to reduce their complexity, see equation \eqref{lem:structmeas:gradientmeas3D}.
More importantly, the equations of Lemma \ref{lemma:H-measure_wave_equations} result in a decomposition of the H-measures into Dirac measures on six discrete directions of oscillation.
It is analogous to the decomposition of the strain into functions of one variable that is central to proving the rigidity of twins, see the papers \cite{DM95,CO09,CO12}.
%that is analogous to the central decomposition of the strain into functions of one variable used in
%The resulting structure 
%At least after having seen that these constraints imply the decomposition in Lemma \ref{lemma: decomposition}, it is not surprising that in this instance the H-measures are supported on the directions $\nu \in N$, see Subsection \ref{subsec:notation} for the definition.
%In particular, we get the same combinatorics in both cases as can be seen by comparing Lemma \ref{lemma: decomposition} with Equation \eqref{3D:lem:decompostion_of_gradient_measure} below.

We also get an expression for the mass of the H-measures $\mu_i^{(0)}$ in equation \eqref{sum_A_old}, although some post-processing in Corollary \ref{cor:diff_incl} will get rid of the austenitic contributions.
Finally, inequality \eqref{sigma_delta_leq_sigma} is a result of $\mu_i^{(\delta)}$ involving a convolution.

\begin{lemma}\label{lem:structmeas3D}
 For $\delta \geq 0$ there exist non-negative measures $\sigma_i\supdelta$ on $\Omega \times \Sph^2$ for $i\in \{1,2,3\}$ and measurable, non-negative functions $A\supdelta_{[\nu]} \in L^\infty(\Omega)$ for $\nu \in N$ such that the following hold:
 For all $\psi \in C_c(\Omega)$ and $a \in C(\Sph^2)$ we have
 \begin{align}
  \mu_i\supdelta(v,w;\psi\otimes a) & = \sigma_i\supdelta\left(\psi \otimes (v\cdot\xi)(w\cdot \xi) a\right),\label{lem:structmeas:gradientmeas3D}\\
  \sigma_i\supdelta(\psi \otimes a) & = \int_\Omega  \psi(x) \sum_{\nu \in N_{i-1} \cup N_{i+1}} A_{[\nu]}\supdelta(x) \delta_{[\nu]}(a)  \intd x, \label{3D:lem:decompostion_of_gradient_measure}
 \end{align}
 where $\delta_{[\nu]}$ is defined as
 \[\delta_{[\nu]}:=\frac{1}{2}\left(\delta_{\frac{\nu}{|\nu|}}+\delta_{-\frac{\nu}{|\nu|}}\right)\numberthis \label{def:dirac_on_projective}\]
 for $\nu \in N$.
 Furthermore, we have
 \begin{align}
 	\sum_{\nu \in N_{i+1} \cup N_{i-1}} A_{[\nu]}^{(0)}& \equiv 18\,\theta_i(1-\theta_i) - 12\, \theta_0\theta_i + 2\, \theta_0(1-\theta_0),& \label{sum_A_old}\\
 	\sigma_i\supdelta & \leq \sigma_i  \label{sigma_delta_leq_sigma}
 \end{align}
 for $\delta > 0$.
\end{lemma}

So far we only proved that oscillations are restricted to the six twinning directions.
However, in order to argue that the microstructures locally are twins, we have to make sure that at almost all points in space there is oscillation in at most one direction.
This is a consequence of the rigidity result by Capella and Otto \cite{CO12} and the behavior of the energy under rescaling:
Setting
\begin{align}\label{h-measure:rescaling}
	r \hat x = x\text{, }\hat u ( \hat x) =  r u (x)\text{, } \hat \chi (\hat x)   = \chi(x)\text{, } r \hat \eta  = \eta
\end{align}
we obtain
\[ E_{\hat \eta}(\hat u, \hat \chi) = r^{-3 + \frac{2}{3}}  E_\eta(u,\chi),\numberthis \label{h-measure:energy_rescaling}\]
which very naturally leads to the expected fractal dimension $3-2/3$ of the set of macroscopic interfaces.
The same argument was used in the companion paper \cite{simon2017rigidity} to establish the limiting non-convex differential inclusion.

% The lemma is somewhat reminiscent of partial regularity results in the theory of nonlinear partial differential equations with the noteable difference that here we have no information about the topological properties of the exceptional set.

\begin{prop}\label{prop:disjoint_supports}
 For $\nu, \tilde \nu \in N$ with $\nu \neq \tilde \nu$ we have
\begin{align}\label{brot}
  \left(A^{(0)}_{[\nu]} A^{(0)}_{[\tilde \nu]}\right)(y) & = 0
\end{align}
 in the sense of Lebesgue points for all $y \in \Omega\setminus S$, where the set 
 \[S := \left\{y \in \Omega:\limsup_{r\to 0} r^{-3 + \frac{2}{3}} (E_{elast} + E_{inter})\left(\overline{\ball{y}{r}}\right) >0\right\}\] satisfies 
\[\operatorname{dim}_H S \leq 3-\frac{2}{3}.\]
 Furthermore we have
 \[\theta_0 \in \{0,1\}\]
 almost everywhere.
\end{prop}

As an easy consequence of this proposition, we can refine the statement of Lemma \ref{lem:structmeas3D}.

\begin{cor}\label{cor:diff_incl}
 For each $\nu \in N_i$  with $i\in \{1,2,3\}$ there exist $\chi_{[\nu]} : \Omega \to \{0,1\}$ measurable such that
 \begin{align}
 	\chi_{[\nu]}\chi_{[\tilde \nu]} & \equiv 0 \text{ for } \tilde \nu \in N\setminus \{\nu\}, \label{disjoint_supports}\\
 	A^{(0)}_{[\nu]} & \equiv 18\, \theta_i(1-\theta_i)\chi_{[\nu]},\label{relation_A_chi}\\
 	\sum_{\nu \in N_{i+1} \cup N_{i-1}} \chi_{[\nu]} & \equiv \chi_{\{\theta_i \neq 0,1\}} \chi_{\{\theta_0 = 0\}}. \label{support_sum_chi}
 \end{align}
\end{cor}

\subsection[The transport property and accuracy of the approximation]{The transport property and accuracy of the approximation\sectionmark{The transport property}}
\sectionmark{The transport property}
\label{subsec:H_measure_transport}

In its simplest form, the transport property states the following:
Let $u_n : \Omega \to \R$ be such that $v_{1,n}:= u_n \warr 0$ and $v_{2,n}:=\partial_1 u_n \warr  0$ in $L^2$ and let $\mu(i,j; \bullet \otimes \bullet)$ be the associated H-measures for $i,j \in\{1,2\}$.
Then we have
\[ \mu(1,1;\partial_l |\psi|^2\otimes a) = -2 \mu(1,2;|\psi|^2\otimes a).\]
In particular, we see that we need control of a derivative of the sequence.
However, not having a derivative to spend since we are already considering the gradient forces us to regularize the sequence:
Recall that we set $u^{(\delta)}_{\eta}= \varphi_{\delta\eta^{\frac{1}{3}}}\ast u_\eta$ for $\delta>0$ in Subsection \ref{subsec:H-measure_existence}.

As a result, we have to investigate how well the regularized H-measures represent the microstructure, which will boil down to how much mass they retain by inequality \eqref{sigma_delta_leq_sigma} and the fact that there can locally only be at most one direction of oscillation.
The proof straightforwardly uses the interfacial energy to control the difference between the sequence $\chi_\eta$ and its convolution.

\begin{lemma}\label{lemma:mass_in_approx_converges}
 There exist non-negative measurable functions $\tau_i\supdelta$ on $\Omega$ for $i=1,2,3$ such that
 \begin{align}\label{delta_weight}
  \sigma_i\supdelta & = \tau_i\supdelta \left(\sum_{\nu \in N_{i+1} \cup N_{i-1}} \chi_{[\nu]}  \delta_{[\nu]} \right)\Leb^3
 \end{align}
and
\[18\,\theta_i(1-\theta_i)-36\delta\frac{DE_{inter}}{D\Leb^3}\leq \tau_i\supdelta\leq 18\,\theta_i(1-\theta_i) \text{ } \numberthis\label{bounds_tau}\]
in $\Leb^3$-almost all points.
\end{lemma}

Next, we come to the transport property itself.
It controls how on a twin the mass of the H-measures changes in directions normal to the direction of lamination.
By equation \eqref{sum_A_old} this also restricts the volume fractions the behavior of the volume fractions.
Note that the transport property takes the form of a differential inequality associated to the ill-posed ODE
\[f'= C f^\frac{1}{2}\]
for $f\geq 0$ and $C>0$, which we will later exploit in Lemma \ref{lemma:differential_inequality}.

\begin{prop}\label{prop:transport_property}
There exists a universal constant $C>0$ with the following property:
For each $\delta >0$, $i=1,2,3$ and $\nu \in N_{i+1} \cup N_{i-1}$ let $d \in \Sph^2$ with $d\cdot \nu =0$.
Then we have $\partial_d\left(\tau_i\supdelta \chi_{[\nu]}\right) \in L^2(\Omega)$ with the estimate
 \begin{align}
  \left|\partial_d\left(\tau_i\supdelta \chi_{[\nu]}\right)\right| & \leq C \frac{1}{\delta}\left(\tau_i\supdelta \chi_{[\nu]}\right)^{\frac{1}{2}}\left(\frac{DE_{elast}}{D\Leb^3}\right)^{\frac{1}{2}}.
    \label{3D_def_meas_mixed}
 \end{align}
\end{prop}

In the next step we "interpolate" the above two statements to obtain Besov regularity of twins in directions along the twin.
The main assumption is that there are either no oscillations or at least a certain amount of them, which boils down to the volume fractions of the martensite variants either being zero or bounded away from it.

\begin{lemma}\label{lemma:Besov_incomplete}
	There exists a universal constant $C>0$ with the following property:
	
	Let $i\in \{1,2,3\}$.
	Let there exist $\eps > 0 $ such that $18 \theta_i (1-\theta_i) \geq \eps$ almost everywhere on the set $\{0<\theta_i<1\}$.
	Let $U \subset\subset \Omega$ be an open subset.
	
	Then for $\nu \in N_{i+1} \cup N_{i-1}$, $d \in \Sph^2$ with $\nu \cdot d = 0$ and $0<h<\dist(U,\partial\Omega)$ we have
	\[\int_U |\partial_d^h\chi_{[\nu]} |\intd x \leq  C \eps^{-1} \left(E_{inter}^\Leb(U_h)\right)^\frac{2}{3}\left(E_{elast}^\Leb(U_h)\right)^\frac{1}{3} h^\frac{2}{3},\]
	where $E^{\Leb}_{elast}$ and $E^{\Leb}_{inter}$ are given by definition \eqref{Lebesgue_densities_notation}.
\end{lemma}

The proof relies on the following easy consequence of the differential inequality, which we state separately to avoid redundant arguments.
Note that it is optimized for quick applicability in our setting and not for maximal generality.

\begin{lemma}\label{lemma:differential_inequality}
	Let $f : [0,1] \to [0,1]$ be continuous with $f(0)=0$.
	Furthermore, let it satisfy the differential inequality
	\[f'\leq f^\frac{1}{2} g^\frac{1}{2}\]
	almost everywhere for an integrable function $g: [0,1] \to [0,\infty)$.
	
	Then we have the estimate
%	\[f(t) \leq t^2 \left(\dashint_0^t g(s) \intd s\right)^2 .\]
%	In particular, if $g \in L^2(0,1)$ we also have
	\[f(t) \leq t^2 \dashint_0^t g(s) \intd s .\]
\end{lemma}

\section{Proofs}\label{sec:proofs}
\subsection{Existence of the H-measures}

\begin{proof}[Proof of Lemma \ref{lemma:H_measure_non-dense_set}]
For $\delta_1,\delta_2>0$ we have the estimate
\[\left|\left|\varphi_{\delta_1\eta^{\frac{1}{3}}}-\varphi_{\delta_2\eta^{\frac{1}{3}}}\right|\right|_{L^1(\R^3)} = \left|\left|\varphi_{\frac{\delta_1}{\delta_2}}-\varphi_1\right|\right|_{L^1(\R^3)}\lesssim \left|1-\frac{\delta_1}{\delta_2}\right|.\]
Thus for $i,j,k=1,2,3$; $\psi_1,\psi_2\in C_c(\Omega)$ and $a\in C(\Sph^2)$ we have
\begin{align*}& \bigg|\int_{\R^3}\F\left(\psi_1\left(\partial_j u^{(\delta_1)}_{i,\eta}-\partial_j u_i\right)\right)\F^*\left(\psi_1\left(\partial_k u^{(\delta_1)}_{i,\eta}-\partial_k u_i\right)\right)a\intd \xi\\
& \qquad - \int_{\R^3}\F\left(\psi_1\left(\partial_j u^{(\delta_2)}_{k,\eta}-\partial_j u_i\right)\right)\F^*\left(\psi_1\left(\partial_k u^{(\delta_2)}_{i,\eta}-\partial_k u_i\right)\right)a\intd \xi\bigg|\\
 \lesssim & ||\psi_1||_\infty||\psi_2||_\infty||a||_\infty \sup_{\eta,\delta}\left(||\nabla u_\eta^{(\delta)}||_{L^2} \right) ||\nabla u_\eta^{(\delta_1)} - \nabla u_\eta^{(\delta_2)}||_{L^2}\\
 \lesssim & ||\psi_1||_\infty||\psi_2||_\infty||a||_\infty \left(\sup_{\eta}||\nabla u_\eta||_{L^2} \right)^2 \left|1-\frac{\delta_1}{\delta_2}\right|.
\end{align*}
As this implies convergence to zero as $|\delta_1 - \delta_2| \to 0$ uniformly in $\eta$ we see that the claim holds.
% which proves that $\mu_i^{(\delta_j)}(e_j,e_k;\bullet \otimes \bullet)$ is a Cauchy sequence in the space of measures on $\Omega\times\Sph^1$ if the H-measures exist and $\delta_j \to \delta$ for some $\delta>0$.
% The same estimate can also be used to prove that the limiting measure is in fact the H-measure associated to $\partial_j u^{(\delta)}_{i,\eta}-\partial_j u_i$ and $\partial_k u^{(\delta)}_{i,\eta}-\partial_k u_i$.
\end{proof}

\subsection{Structure of the H-measures}
\begin{proof}[Proof of Lemma \ref{lemma:H-measure_wave_equations}]
We first deal with the case $\delta = 0$.
 Throughout the proof $h_\eta$ is a generic sequence of vector fields satisfying the desired bound which can change from line to line.
 By symmetry it is sufficient to prove the equations involving $u_1$.
 We calculate
 \begin{align*}
  \partial_{[111]}\partial_{[\overline 111]}  & = -\partial^2_1 + \partial_1\partial_2 + \partial_1\partial_3 - \partial_1\partial_2 + \partial^2_2 + \partial_2\partial_3 - \partial_1\partial_3  + \partial_2\partial_3  + \partial^2_3 \\
  & = -\partial_1^2 +\partial_2^2 + \partial_3^2 + 2 \partial_2\partial_3
 \end{align*}
 and, similarly,
  \[\partial_{[1\overline11]}\partial_{[ 11\overline1]}  = \partial_1^2 - \partial_2^2 - \partial_3^2 + 2 \partial_2\partial_3.\]
 Setting the strain space $S = \{e \in \R^{3\times 3}: e \text{ diagonal}, \tr e = 0\}$ and noting that $e_0, \ldots, e_3 \in S$ we get for $\delta>0$ that
 \[\dist^2\left(\frac{1}{2}(Du_\eta + Du_\eta^T),S\right) \leq \left|\frac{1}{2}\left(Du + Du^T\right) -\sum_{i=1}^3\chi_i e_i\right|^2. \] 
 Consequently, we see
 \begin{align*}
  (-\partial_1^2 +\partial_2^2 + \partial_3^2)u_{1,\eta} & = -\partial_1^2 u_{1,\eta} -\partial_2\partial_1 u_{2,\eta} -\partial_3\partial_1u_{3,\eta} + \Div h_\eta \\
  & = -\partial_1 \operatorname{tr} Du_\eta + \Div h_\eta \\
  & = \Div h_\eta.
 \end{align*}
 We also obtain
 \[\partial_2\partial_3 u_{1,\eta} = - \partial_2\partial_1 u_{3,\eta} + \Div h_\eta = \partial_1\partial_3 u_{2,\eta} + \Div h_\eta = - \partial_2\partial_3 u_{1,\eta} + \Div h_\eta\]
 and because the derivatives appear on both sides with opposite signs we have \[\partial_2\partial_3 u_{1,\eta} = \Div h_\eta.\]
 For $\delta >0$ we only have to use that convolution and differentiation commute.
%  the only change is that one additionally has to use the estimate
% \[\int_\Omega \psi^2(x) \dist\left(\frac{1}{2}\left(Du\supdelta_\eta + (Du\supdelta_\eta)^T\right),V\right)  \leq \int_\Omega \dist\left(\frac{1}{2}(Du_\eta + Du_\eta^T),V\right) \]
% for $\psi \in C_c(\Omega;[0,1])$ such that $\supp(\psi) + \ball{0}{\delta} \subset \Omega$.
\end{proof}

\begin{proof}[Proof of Lemma \ref{lem:structmeas3D}]
\textit{Step 1: Gradient H-measures.}\\
Equation \eqref{lem:structmeas:gradientmeas3D} is simply the characterization of gradient H-measures \cite[Lemma 3.10 first part]{Tartar90}.
We will however briefly give the argument:
As the generating sequence for $\mu\supdelta_i$ is curl-free, the localization principle for H-measures \cite[Theorem 1.6]{Tartar90} implies
\[\mu\supdelta_i(E_j,E_k;\bullet \otimes \xi_m \bullet) = \mu\supdelta_i(E_m,E_k;\bullet \otimes \xi_j \bullet)\]
for all $i,j,k,m = 1,2,3$.
For all $i,j,k =1,2,3$ we consequently have
\[\mu\supdelta_i(E_j,E_k; \bullet \otimes \bullet) = \sum_{m=1}^3 \mu\supdelta_i(E_j,E_k; \bullet \otimes \xi_m^2\bullet) = \sum_{m=1}^3 \mu\supdelta_i(E_m,E_k; \bullet \otimes \xi_j\xi_m\bullet).\]
As the measure $\mu\supdelta_i$ is hermitian non-negative \cite[Corollary 1.2]{Tartar90} we see that
\begin{align*}
  \sum_{m=1}^3 \mu\supdelta_i(E_m,E_k; \bullet \otimes \xi_j\xi_m\bullet) & =  \sum_{m=1}^3 \mu\supdelta_i(E_k,E_m; \bullet^* \otimes \xi_j\xi_m\bullet^*)^* \\
  & = \sum_{m=1}^3 \mu\supdelta_i(E_m,E_m; \bullet^* \otimes \xi_j\xi_k\bullet^*)^*\\
  & = \sum_{m=1}^3 \mu\supdelta_i(E_m,E_m; \bullet \otimes \xi_j\xi_k\bullet)\\
  & = \sigma\supdelta_i(\bullet \otimes \xi_j\xi_k\bullet)
\end{align*}
for $\sigma\supdelta_i := \sum_{m=1}^3 \mu\supdelta_i(E_m,E_m; \bullet \otimes \bullet)$.
In particular, the measure $\sigma\supdelta_i$ is non-negative in the sense that if $\psi \in C_c(\R^3;\R_{\geq 0})$ and $a\in C(\Sph^2;\R_{\geq 0})$ we have
$\sigma\supdelta_i(\psi\otimes a) \geq 0$.
% By the characterization of gradient H-measures, see e.g. the first part of \cite[Lemma 3.10]{Tartar90}, there exists a non-negative measure $\nu$ on $\Omega \times \Sph^2$ with $\mu(E_i,E_j;\bullet \otimes \bullet) = \nu(\bullet \otimes \xi_i\xi_j \bullet)$ for $i,j = 1,2,3$.
Since both sides of equation $\eqref{lem:structmeas:gradientmeas3D}$ are bilinear and they agree on a basis of $\R^3$ we must have equality for all $v,w\in \R^3$.

\textit{Step 2: Structure of the Fourier variable part.}\\
Combining the localization principle \cite[Theorem 1.6]{Tartar90} with the first equation of Lemma \ref{lemma:H-measure_wave_equations} we see that
\[\mu\supdelta_1([111],E_j;\bullet \otimes \xi \cdot [\overline111]\bullet) = 0\]
for all $j=1,2,3$.
Writing this in terms of $\sigma_i$ and replacing $\bullet$ by $\xi_j\bullet$ this reads
\[\sigma\supdelta_1\Big(\bullet\otimes (\xi \cdot [\overline111])(\xi \cdot [111])\xi_j^2 \bullet\Big) = 0.\]
Summation in $j$ yields
\[\sigma\supdelta_1\Big(\bullet\otimes (\xi \cdot [\overline111])(\xi \cdot [111]) \bullet\Big) = 0.\]
Using the second equation of Lemma \ref{lemma:H-measure_wave_equations} we instead get
\[\sigma\supdelta_1\Big(\bullet\otimes (\xi \cdot [1\overline11])(\xi \cdot [11\overline1]) \bullet\Big) = 0.\]
In particular, for every $\psi \in C_c(\R^3;\R_{\geq>0})$ we have
\begin{align*}
	& \quad  \supp(\sigma\supdelta_1(\psi \otimes \bullet)) \\
	& \subset \left(\{\xi \cdot [\overline111] = 0\} \cup \{\xi \cdot [111] = 0\} \right) \cap \left(\{\xi \cdot [1\overline11] = 0\} \cup \{\xi \cdot [11\overline1] = 0\} \right) \cap \Sph^2\\
	& = \pm N_2 \cup \pm N_3,
\end{align*}
where the last step is a straightforward consequence of definition \ref{def:normals}.
Consequently, the measure $\supp(\sigma\supdelta_1(\psi \otimes \bullet)) $ is a linear combination of Dirac measures supported on the set $\pm N_2 \cup \pm N_3$, where the coefficients are given by integrating $\psi$ against Radon measures on $\Omega$.

Because for real valued functions $f$ on $\R^3$ we have $\overline{\mathcal{F}f(\xi)}=\mathcal{F}f(-\xi),$
we see that the measure $\sigma\supdelta_1(\psi \otimes\bullet)$ is invariant under reflection in $\xi$ due to being non-negative and thus real-valued.
Hence there exist non-negative Radon measures $\omega\supdelta_{1,[\nu]}$ on $\Omega$ for $\nu\in N_2\cup N_3$ such that
\[\sigma\supdelta_1=\sum_{\nu \in N_2 \cup N_3} \omega\supdelta_{1,[\nu]} \otimes \delta_{[\nu]}.\]
We can also relate the ``cumulative'' gradient $H$-measure $\sigma_1\supdelta$ exclusively to the $H$-measure $\mu_1$ associated to the corresponding diagonal entry $e(u)_{11}$ of the strain:
 For $\psi \in C_c(\Omega,\R_{\geq 0})$ and $a\in C(\Sph^2;[0,1])$ we see using $ \frac{1}{2}\nu^2 = \nu_1^2$ for $\nu \in N_2 \cup N_3$ and the characterization of gradient H-measures \eqref{lem:structmeas:gradientmeas3D} that
  \[\frac{1}{2}\sigma_1\supdelta(\psi^2 \otimes a) = \frac{1}{2}\sigma_1\supdelta(\psi^2\otimes \xi^2 a) =\sigma_1\supdelta(\psi^2\otimes \xi_1^2 a) = \mu_1\supdelta(E_1,E_1;\psi^2 \otimes a).\numberthis \label{representation_gradient_h_measure}\]

Using similar arguments, we see that also for $i=2,3$ there exist Radon measures $\omega\supdelta_{i,[\nu]}$ on $\Omega$ for $\nu\in N_{i+1}\cup N_{i-1}$ such that
\[\sigma\supdelta_i=\sum_{\nu \in N_{i+1} \cup N_{i-1}} \omega\supdelta_{i,[\nu]} \otimes \delta_{[\nu]} = 2 \mu_i(E_i,E_i;\bullet \otimes \bullet) .\]
% Note that $\omega_{i,i+1}^{(\pm)}$ appear in $\nu_i$ and $\nu_{i+1}!$
% Also, recall that we defined \[\delta_{j,j+1}^{(+)}=\frac{1}{2}\left(\delta_{\left(\xi_{j-1}=0,\xi_j=\frac{1}{\sqrt 2},\xi_{j+1}=\frac{1}{\sqrt 2}\right)}+\delta_{\left(\xi_{j-1}=0,\xi_j=-\frac{1}{\sqrt 2},\xi_{j+1}=-\frac{1}{\sqrt 2}\right)}\right)\] and \[\delta_{j,j+1}^{(-)}=\frac{1}{2}\left(\delta_{\left(\xi_{j-1}=0,\xi_j=\frac{1}{\sqrt 2},\xi_{j+1}=-\frac{1}{\sqrt 2}\right)}+\delta_{\left(\xi_{j-1}=0,\xi_j=-\frac{1}{\sqrt 2},\xi_{j+1}=\frac{1}{\sqrt 2}\right)}\right).\]
\textit{Step 3: For $\{i,j,k\} = \{1,2,3\}$ we have $\omega\supdelta_{i,[\nu]}= \omega\supdelta_{j,[\nu]}$ for $\nu \in N_k$. In particular, we may write $\omega_{[\nu]}\supdelta$ instead.}\\
In order to keep the notation simple we will only deal with the case $i=1$, $j=2$ and $k=3$.
All others work similarly.
Let $\nu \in N_3$.
Let $\psi \in C_c(\Omega)$ and let $a \in C(\Sph^2;[0,1])$ be such that $a(\pm\nu) = 1$ and $a(\pm \tilde \nu) = 0$ for $\tilde \nu \in N\setminus \{\nu\}$.
As all limiting strains $e_i$ are trace-free we get that
\[\partial_1u_1\supdelta + \partial_2u\supdelta_2 + \partial_3u\supdelta_3 \to 0 \text{ in } L^2.\]
Consequently, we get that
\begin{align*}
	& \quad \mu_1\supdelta(E_1,E_1;|\psi|^2\otimes a) \\
	& = \lim_{\eta \to 0} \int_{\R^ 2}\left|\F\left(\psi\left(\partial_2 u\supdelta_{2,\eta} +\partial_3 u\supdelta_{3,\eta}-(\partial_2 u_2 + \partial_3 u_3)\right)\right)\right|^2 a\left(\frac{\xi}{|\xi|}\right)\intd \xi.
\end{align*}
Expanding the square we see that all terms involving $\partial_3 u_{3,\eta}\supdelta - \partial_3 u_3$ drop out since
\begin{align*}
	\lim_{\eta \to 0} \int_{\R^ 2}\left|\F\left(\psi\left(\partial_3 u\supdelta_{3,\eta}- \partial_3 u_3\right)\right)\right|^2 a\left(\frac{\xi}{|\xi|}\right)\intd \xi
	& = \mu_3\supdelta\left(E_3,E_3;|\psi|^2\otimes a\right) \\
	& = \sigma_3\supdelta\left(|\psi|^2 \otimes \xi_3^2 a \right)\\
	& = 0
\end{align*}
due to $\supp \sigma_3(\psi\otimes \bullet) \subset \pm N_1 \cup N_2$ and $a(\pm \tilde \nu)$ for $\tilde \nu \in N \setminus \{\nu\}$.
As a result we get
\[\mu_1\supdelta(E_1,E_1;|\psi|^2\otimes a) = \mu_2\supdelta(E_2,E_2;|\psi|^2\otimes a),\]
which using the choice of $a$ localizing at $\pm \nu$ and the representation \eqref{representation_gradient_h_measure}	 implies
\[\frac{1}{2}\omega_{1,[\nu]}\supdelta(\bullet) = \frac{1}{2}\sigma_1\supdelta(\bullet\otimes a) = \mu_1\supdelta(E_1,E_1;\bullet\otimes a) = \mu_2\supdelta(E_2,E_2;\bullet\otimes a) = \frac{1}{2}\omega_{2,[\nu]}\supdelta(\bullet).\]

\textit{Step 4: Absolute continuity of $\omega_{[\nu]}^{(0)}$ for $\nu \in N$ and equation \eqref{3D:lem:decompostion_of_gradient_measure} for $\delta =0$.}\\
 Note that we will drop the superscript for the duration of this step.
 Furthermore, we only deal with the case $\nu \in N_2 \cup N_3$.
 The case $\nu \in N_1$ works the same.

  Recalling the representation \eqref{representation_gradient_h_measure}, the definition of H-measures and $0\leq a \leq 1$ we obtain
   \[\frac{1}{2}\sigma_1(\psi^2 \otimes a) = \lim_{\eta \to 0} \int_\R^3 a |\F (\psi \partial_1(u_{1,\eta}-u_1))|^2 \intd \xi \leq \lim_{\eta \to 0} \int_\R^3  |\F (\psi \partial_1 (u_{1,\eta}-u_1))|^2 \intd \xi\]
  with equality for $a\equiv 1$.
  An application of Parseval's theorem implies
  \begin{align}\label{Parseval}
    \frac{1}{2}\sigma_1(\psi^2 \otimes a) \leq \lim_{\eta \to 0} \int_\Omega  |\psi \partial_1 (u_{1,\eta}-u_1)|^2 \intd x.
  \end{align}

  Using equations \eqref{relation_Du_chi} and \eqref{relation_Du_theta} we can relate this limit to the limiting volume fraction $\theta_i$ by observing
  \[ \lim_{\eta \to 0} \int_\Omega \psi^2 \left(\partial_1 u_{1,\eta} - \partial_1 u_1\right)^ 2 \intd \Leb^3 = \lim_{\eta \to 0}\int_\Omega \psi^2 \left(3\chi_1 + \chi_0 -3\theta_1 - \theta_0\right)^2 \intd \Leb^3. \]
  Expanding the square and using the fact that $\chi_0$ and $\chi_1$ are characteristic functions of disjoint sets we see that the right-hand side equals
  \begin{align*}
    & \quad \lim_{\eta \to 0}\int_\Omega \psi^2 \left(9\chi_1 - 6 \chi_1(3\theta_1 + \theta_0) + \chi_0 - 2 \chi_0(3\theta_1+\theta_0) + (3\theta_1 + \theta_0)^2\right) \intd \Leb^3\\
    & = \int_\Omega  \psi^2\left( 9 \theta_1(1-\theta_1) - 6\theta_0\theta_1 + \theta_0(1-\theta_0) \right)\intd \Leb^3.
%   =\lim_{\eta \to 0}\int_\Omega 9 \psi^2 \left(\chi_1 - \theta_1^2\right) \intd \Leb^3
%   = \int_\Omega 9 \psi^2 \theta_1(1-\theta_1) \intd \Leb^3.
  \end{align*}
  Altogether we proved
  \begin{align}\label{3D:limit_Du^2}
   \lim_{\eta \to 0} \int_\Omega \psi^2 \left(\partial_1 u_{1,\eta} - \partial_1 u_1\right)^ 2 \intd \Leb^3 = \int_\Omega  \psi^2\left( 9 \theta_1(1-\theta_1) - 6\theta_0\theta_1 + \theta_0(1-\theta_0) \right)\intd \Leb^3.
  \end{align}   
  Using $a$ to localize at the directions $\pm \nu \in N_2 \cup N_3$ where $\sigma_i$ may concentrate and combining inequality \eqref{Parseval} with the convergence \eqref{3D:limit_Du^2} we see that $\omega_{[\nu]}$ must be absolutely continuous w.r.t.\ the measure $\Leb^3$ with some density $A_{[\nu]}^{(0)}$.
  When instead using $a\equiv 1$ the estimates turn into the identity
  \[\sum_{\nu \in N_{i+1} \cup N_{i-1}} A_{[\nu]}^{(0)}  \equiv 18\,\theta_i(1-\theta_i) - 12\, \theta_0\theta_i + 2\, \theta_0(1-\theta_0).\]
  
\textit{Step 5:  We have $\sigma_i\supdelta \leq \sigma_i$ as measures for $i=1,2,3$ and $\delta > 0$. In particular, the functions $A_{[\nu]}\supdelta \in L^\infty(\Omega)$ exist such that equation \eqref{3D:lem:decompostion_of_gradient_measure} holds.}\\
Let $\psi \in C_c(\Omega;\R)$.
First note that 
\[\psi \varphi_{\delta\eta^\frac{1}{3}}\ast\nabla\left(u_{1,\eta}-u_1\right) -\varphi_{\delta\eta^\frac{1}{3}}\ast\left(\psi  \nabla\left(u_{1,\eta}-u_1\right)\right) \to 0\]
 in $L^2$.
Thus for $a \in C(\Sph^2;\R_{\geq 0})$ and $j=1,2,3$ we can calculate, exploiting the fact $\left|\left|\F\left(\varphi_{\delta\eta^\frac{1}{3}}\right)\right|\right|_\infty \leq \left|\left|\varphi_{\delta\eta^\frac{1}{3}}\right|\right|_{L^1} =1$ along the way, that
\begin{align*}
 \mu_1\supdelta\left(E_j,E_j;\psi^2\otimes a\right) & = \lim_{\eta \to 0} \int a \left|\F \left(\varphi_{\delta\eta^\frac{1}{3}}\ast\psi  \partial_j\left(u_{1,\eta}-u_1\right)\right)\right|^2 \intd \xi \\
 & = \lim_{\eta \to 0} \int a \left|\F\left(\varphi_{\delta\eta^\frac{1}{3}}\right)\right|^2 \left|\F (\psi  \partial_j\left(u_{1,\eta}-u_1\right)\right|^2 \intd \xi\\
 & \leq \lim_{\eta \to 0} \int a \left|\F \left(\psi  \partial_j\left(u_{1,\eta}-u_1\right)\right)\right|^2 \intd \xi\\
 & \leq \mu_1\left(E_j,E_j;\psi^2\otimes a\right).
\end{align*}

%VORSICHT MIT DER NORMALISIERUNG!!! WIE WIRKT SICH DIE AUF DIE EIGENSCHAFTEN DER HMASSE AUS? ICH BENUTZE HIER AUCH DIE FALTUNGSEIGENSCHAFT DER FOURIERTRAFO.
An application of the identity \eqref{representation_gradient_h_measure} yields
\[
  \sigma_1\supdelta\left(\psi^2 \otimes a\right)  = 2\mu_1\supdelta\left(E_j,E_j;\psi^2\otimes a\right)  \leq 2 \mu_1\left(E_j,E_j;\psi^2\otimes a\right)  = \sigma_1\left(\psi^2 \otimes a\right). \qedhere
\]
\end{proof}

\begin{proof}[Proof of Proposition \ref{prop:disjoint_supports}]
Let $y \in \Omega$ and $r>0$ be such that $\ball{y}{r}\subset \Omega$.
By translation invariance we can assume $y=0$.

\textit{Step 1: Applying stability of twins after rescaling.}\\
Setting $\hat x := \frac{x}{r}$ and $\hat \eta :=\frac{\eta}{r}$ we re-scale the displacements and partitions to the unit ball:
Let $\hat u_\frac{\eta}{r}: \ball{0}{1} \to \R^2$ and $\hat \chi_\frac{\eta}{r}: \ball{0}{1} \to \{0,1\}$ be defined as \[\hat u_{\hat \eta}(\hat x):= \frac{1}{r}u_\eta\left(r\hat x\right)\text{, }\hat \chi_{\hat \eta}(\hat x):= \chi_\eta\left(r\hat x\right).\]
The energy of the re-scaled functions is
\begin{align*}
  E_{\hat\eta}(\hat u_{\hat \eta},\hat \chi_{\hat\eta}) & = \hat \eta^{-\frac{2}{3}}\int_\ball{0}{1} \left| e(\hat u_{\hat \eta})-\sum_{i=1}^3 \hat \chi_{i,\hat\eta} e_i\right|^2 \intd \hat x + \hat \eta ^{\frac{1}{3}}|D\hat \chi_{\hat \eta}|(\ball{0}{1})\\
  & = r^{-3 + \frac{2}{3}}E_{\eta}(\ball{0}{r}).
\end{align*}

%To keep the notation simple we additionally assume that $\tilde \nu = \frac{1}{\sqrt 2} (110)$.
%As usual this assumption is non-essential.
By the Capella-Otto rigidity theorem \cite{CO12} there exists a universal radius $0<s<1$ and bounded functions $\hat f_{\nu,\hat \eta}:\ball{0}{1} \to \R$ depending only on $x\cdot \nu$ with $\nu \in N$ such that
\begin{align}\label{kaese}
  \begin{split}
  \min\left\{
    \min_{i=1,2,3; \nu \in N_i} \left\{\left|\left| e(\hat u_\eta ) - \hat f_{\nu,\hat \eta}\,e_{i+1} - (1- \hat f_{\nu,\hat \eta})\, e_{i-1}\right|\right|^2_{L^2(\ball{0}{s})}\right\},
    ||e(\hat u_\eta)||^2_{L^2(\ball{0}{s})}
    \right\}\\
    \lesssim \left(r^{-3 +\frac{2}{3}}E_{\eta}(\ball{0}{r})\right)^\frac{1}{4} + r^{-3 +\frac{2}{3}}E_{\eta}(\ball{0}{r}).
  \end{split}
\end{align}
Just keeping the $i$-th and the $(i+1)$-th diagonal entries of the strain in the inner minimum and the entire diagonal in $||e(\hat u_\eta )||^2$ we see with $(e_i)_{jj} = 1 - 3 \delta_{ij}$ for $i,j\in \{1,2,3\}$ that
\begin{align*}
  \min\Bigg\{
    \min_{i=1,2,3; \nu \in N_i} \Big\{\left|\left|\partial_i \hat u_{i,\hat \eta} -1 \right|\right|^2_{L^2(\ball{0}{s})} 
    & + \left|\left|\partial_{i+1}\hat u_{i+1,\hat \eta} + 3 \hat f_{\nu,\hat \eta}-1 \right|\right|^2_{L^2(\ball{0}{s})}\Big\},\\
    &\qquad \qquad \qquad \qquad \qquad \sum_{i=1}^3||\partial_i \hat u_i||^2_{L^2(\ball{0}{s})}
   \Bigg\}\\
    & \lesssim \left(r^{-3 +\frac{2}{3}}E_{\eta}(\ball{0}{r})\right)^\frac{1}{4} + r^{-3 +\frac{2}{3}}E_{\eta}(\ball{0}{r}).
\end{align*}

Re-scaling back to $\ball{0}{r}$ we get $f_{\nu,\eta}:\ball{0}{r} \to \R$ bounded and depending only on $x\cdot \nu$ for each $\nu \in N$ with
\begin{align}\label{osc_error_sequence}
	\begin{split}
	  \min\Bigg\{
	  \min_{i=1,2,3; \nu \in N_i} \Big\{r^{-3}\left|\left|\partial_i  u_{i, \eta} -1 \right|\right|^2_{L^2(\ball{0}{sr})} 
	  & + r^{-3}\left|\left|\partial_{i+1} u_{i+1,\eta} - f_{\nu, \eta}\right|\right|^2_{L^2(\ball{0}{sr})}\Big\},\\
	  &\qquad \qquad \qquad \qquad  \sum_{i=1}^3r^{-3}||\partial_iu_i||^2_{L^2(\ball{0}{sr})}
	  \Bigg\}\\
	  & \lesssim \left(r^{-3 +\frac{2}{3}}E_{\eta}(\ball{0}{r})\right)^\frac{1}{4} + r^{-3 +\frac{2}{3}}E_{\eta}(\ball{0}{r}).
	\end{split}
\end{align}

%\begin{align}
%  \min\left\{\frac{1}{r^2}||\partial_1u_{1,\eta}-g_{(11),\eta}||^2_{L^2(\ball{0}{sr})},\frac{1}{r^2}||\partial_1u_{1,\eta}-g_{(1\overline 1),\eta}||^2_{L^2(\ball{0}{rs})}\right\} \lesssim \left(r^{-\frac{4}{3}}E_{\eta}(\ball{0}{r})\right)^\frac{1}{4}.\label{osc_error_sequence}
%\end{align}

\textit{Step 2: The H-measure mostly concentrates on the twinning direction in the sense that
	\begin{align*}
  		& \quad \min_{i=1,2,3; \nu \in N_i} \left\{\sum_{\tilde \nu \in N\setminus\{\nu\}} \dashint_{\ball{0}{\frac{sr}{2}}} A_{[\tilde \nu]}^{(0)} \intd x\right\} 
		% &\lesssim \liminf_{\eta \to 0 } \min_{i=1,2,3; \bar \nu \in N_i} \left\{\frac{1}{r^3}\left|\left|\partial_i u_{i, \eta}\right|\right|^2_{L^2(\ball{0}{sr})}  + \frac{1}{r^3}\left|\left|\partial_{i+1} u_{i+1,\eta}- f_{\bar \nu, \eta}\right|\right|^2_{L^2(\ball{0}{sr})}\right\}.
  		\lesssim \left(r^{-3 +\frac{2}{3}}E\left(\overline{\ball{0}{r}}\right)\right)^\frac{1}{4} + r^{-3 +\frac{2}{3}}E\left(\overline{\ball{0}{r}}\right),
	\end{align*}
	where $E\left(\overline{\ball{0}{r}}\right) :=(E_{elast} + E_{inter})\left(\overline{\ball{0}{r}}\right)$.
}\\
As weak convergence of Radon measures is upper semi-continuous on compact sets we may extract a subsequence such that
\[\lim_{\eta \to 0} E_{\eta}(\ball{0}{r}) \leq E\left(\overline{\ball{0}{r}}\right).\]
After extracting yet another subsequence there exist $f_{\nu}: \ball{0}{1}\to \R$ for each $\nu \in N$ such that
\[f_{\nu,\eta} \overset{\ast}{\warr}f_{\nu} \text{ in }L^\infty\]
and we have
\begin{align}\label{osc_error_limit}
  \begin{split}
    & ||\partial_{i+1} u_{i+1}-f_{\bar \nu}||^2_{L^2(\ball{0}{sr})} \leq \liminf_{\eta \to 0} ||\partial_{i+1}u_{i+1,\eta}-f_{\bar \nu,\eta}||^2_{L^2(\ball{0}{sr})}.
  \end{split}
\end{align}

Let $\nu \in N_i$ for $i\in \{1,2,3\}$.
Note that the localization principle for H-measures implies that the support of any H-measure involving $f_{\nu,\eta}-f_{ \nu}$ as a factor is contained in $\{\pm \nu\}$.
Thus for a cut-off function $\psi \in C_c(\Omega; [0,1])$ of $\ball{0}{\frac{s}{2}}$ in $\ball{0}{s}$ and $a_{[\nu]}\in C(\Sph^2;[0,1])$ with $a_{[ \nu]}(\pm \nu) = 1$ and $a_{[\nu]}(\pm \tilde \nu) =0 $ for $\tilde \nu \in N \setminus \{\nu\}$ we get
\begin{align}\label{wurscht}
	\begin{split}
  		& \quad \mu_{i+1}\left(E_{i+1},E_{i+1};\psi\otimes (1-a_{[ \nu]})\right) \\
  		&  \lesssim \liminf_{\eta \to 0} ||\partial_{i+1}(u_{{i+1},\eta}-u_{i+1} )-(f_{ \nu,\eta}-f_{\nu})||^2_{L^2(\ball{0}{sr})}.
	\end{split}	
 \end{align}
Using the representation \eqref{3D:lem:decompostion_of_gradient_measure} of $\sigma_{i+1}$, identity \eqref{representation_gradient_h_measure}, i.e.,
\[\frac{1}{2}\sigma_{i+1}(\bullet\otimes \bullet) = \mu_{i+1}\left(E_{i+1},E_{i+1};\bullet \otimes \bullet\right),\]
and equations \eqref{wurscht} and \eqref{osc_error_limit} we get
\begin{align*}
  \sum_{\tilde \nu \in N_i\cup N_{i-1} \setminus \{\nu\}} \int_{\ball{0}{r}} \psi A_{[\tilde \nu]}\intd x = \sigma_{i+1}(\psi\otimes( 1- a_{[\nu]}) \lesssim \liminf_{\eta \to 0} ||\partial_{i+1}u_{{i+1},\eta}-f_{ \nu,\eta}||^2_{L^2(\ball{0}{sr})}.
\end{align*}
We plug this estimate into the inequality \eqref{osc_error_sequence} along with the crude estimate 
\begin{align*}
  & \sum_{\tilde \nu \in N_{i+1}\cup N_{i-1}} \int_{\ball{0}{r}} \psi A_{[\tilde \nu]}\intd x 
   = \sigma_i(\psi\otimes 1)\\
  & \lesssim \liminf_{\eta\to 0} \min\left\{ \left|\left|\partial_i u_{i, \eta}\right|\right|^2_{L^2(\ball{0}{sr})}, \left|\left|\partial_i u_{i, \eta} -1 \right|\right|^2_{L^2(\ball{0}{sr})} \right\}
\end{align*}
to see
\begin{align*}
  \quad \min_{i=1,2,3; \nu \in N_i} \left\{\sum_{\tilde \nu \in N\setminus\{\nu\}} \dashint_{\ball{0}{\frac{sr}{2}}} A_{[\tilde \nu]}^{(0)} \intd x\right\} 
% &\lesssim \liminf_{\eta \to 0 } \min_{i=1,2,3; \bar \nu \in N_i} \left\{\frac{1}{r^3}\left|\left|\partial_i u_{i, \eta}\right|\right|^2_{L^2(\ball{0}{sr})}  + \frac{1}{r^3}\left|\left|\partial_{i+1} u_{i+1,\eta}- f_{\bar \nu, \eta}\right|\right|^2_{L^2(\ball{0}{sr})}\right\}.
  & \lesssim \lim_{\eta \to 0} \left(r^{-3 +\frac{2}{3}}E_{\eta}(\ball{0}{r})\right)^\frac{1}{4} + r^{-3 +\frac{2}{3}}E_{\eta}(\ball{0}{r})\\
  & \leq \left(r^{-3 +\frac{2}{3}}E\left(\overline{\ball{0}{r}}\right)\right)^\frac{1}{4} + r^{-3 +\frac{2}{3}}E\left(\overline{\ball{0}{r}}\right).
\end{align*}

\textit{Step 3: Prove $A_{[\nu]}^{(0)}A_{[\tilde \nu]}^{(0)} = 0$ for $\nu \neq \tilde \nu$.}\\
As a result of Step 2 we get for $\nu, \tilde \nu \in N$ with $\nu \neq \tilde \nu$ that
\[\dashint_{\ball{0}{\frac{sr}{2}}} A_{[\nu]}^{(0)}A_{[\tilde \nu]}^{(0)} \intd x \lesssim  \left(r^{-3 +\frac{2}{3}}E\left(\overline{\ball{0}{r}}\right)\right)^\frac{1}{4} + r^{-3 +\frac{2}{3}}E\left(\overline{\ball{0}{r}}\right).\]
Reversing the translation to $y=0$ we see that $y$ is a Lebesgue point of the non-negative function $A_{[\nu]}^{(0)}A_{[\tilde \nu]}^{(0)}$ with
\begin{align}
  A_{[\nu]}^{(0)}A_{[\tilde \nu]}^{(0)} = 0
\end{align}
as long as 
\[y \not\in S= \left\{y \in \Omega:\limsup_{r\to 0} r^{-3 + \frac{2}{3}} (E_{elast} + E_{inter})\left(\overline{\ball{y}{r}}\right) >0\right\}.\]
By standard covering arguments one can see that $\operatorname{dim}_H S\leq 3 - \frac{2}{3}$, which concludes the proof of the first part of the statement.

\textit{Step 4: We have $\theta_0 \in \{0,1\}$ for almost all Lebesgue points of $\theta_0$.}\\
The argument is very similar to Steps 1 and 3.
Instead of using the result of Capella and Otto in the form of estimate \eqref{kaese} we apply it as
\[\min \left\{\int_{\ball{0}{s}} |\chi_{0,\eta}| \intd \hat x,\int_{\ball{0}{s}} |\chi_{0,\eta} -1 | \intd \hat x \right\} \lesssim  r^{-3 + \frac{2}{3}}E_{\eta}(\ball{0}{r}).\]
Re-scaling the left-hand side to $\ball{y}{sr}$, taking the limit $\eta \to 0$ and using $\Leb^3(S)=0$ we get the desired statement.
Note that we can only get rid of the minimum in the localization $r\to 0$ if we a priori know $y$ to be a Lebesgue point of $\theta_0$. 
% % satisfy the properties \eqref{disjoint_supports}-\eqref{relation_A_chi}.
%Equations \eqref{support_sum_chi} and \eqref{relation_A_chi} finally follow from property \eqref{sum_A}.
%Thus any point not contained in this set of Hausdorff dimension at most $\frac{4}{3}$ is a Lebesgue point of $\omega_1$ or $\omega_2$ and the taken value is zero. It is elementary to see that $\lambda \in \{0,1\}$ for $\mathcal L^2$-a.e.\ $x$ in $\Omega$.
%The fact that $\lambda = \lambda\supdelta$ is an immediate consequence of $\lambda \in \{0,1\}$ and $\nu\supdelta \leq \nu$.
\end{proof}

\begin{proof}[Proof of Corollary \ref{cor:diff_incl}]
Let $\chi_{[\nu]}:= \chi_{\{A^{(0)\}}_{[\nu]} >0\}}$.
Equation \eqref{disjoint_supports}, namely $\chi_{[\nu]}\chi_{[\tilde\nu]}\equiv0$ for $\nu \neq \tilde \nu$, is an immediate consequence of equation \eqref{brot}.
To prove
\[A_{[\nu]}^{(0)} = 18 \theta_i(1-\theta_i)\chi_{[\nu]},\]
which is equation \eqref{relation_A_chi}, observe that 
\begin{align}\label{quark}
  \theta_j=0 \text{ for }j=1,2,3 \text{ almost everywhere on the set }\{\theta_0 = 1\} 
\end{align}
 due to $\sum_{i=0}^3 \theta_i \equiv 1$ and $0\leq \theta_i \leq  1$ for $i=0,\ldots,3$.
Therefore equation \eqref{sum_A_old} turns into
\[\sum_{\nu \in N_{i+1} \cup N_{i-1}} A_{[\nu]}^{(0)} = 18 \theta_i(1-\theta_i),\]
which implies \eqref{relation_A_chi} by equation \eqref{brot}.

This identity together with observation \eqref{quark} implies both that $\chi_{[\nu]} \equiv 0$ on $\{\theta_0 =1\}$ for all $\nu \in N$ and that for almost every $x\in \{0<\theta_i <1\}$ there exists some $\nu \in N_{i+1} \cup N_{i-1}$ such that $\chi_{[\nu]}(x)=1$.
Consequently, we have equation \eqref{support_sum_chi}, namely
\[\sum_{\nu \in N_{i+1} \cup N_{i-1}} \chi_{[\nu]}  \equiv \chi_{\{\theta_i \neq 0,1\}} \chi_{\{\theta_0 = 0\}}.\qedhere\]
%
%To prove the last part we assume that $\theta_0 \equiv 0$.
%Then for almost all $x\in \Omega$ the statement $\chi_{[\nu]}(x) =1$ for $\nu \in N_i$ and $i\in \{1,2,3\}$ implies $\chi_{[\tilde \nu]}(x) =0$ for $\tilde \nu \in N_{i+1} \cup N_{i-1}$ by our first insight of this proof.
%Therefore the above equation gives
%\[\chi_{\{\theta_i \neq 0,1\}} = \sum_{\tilde \nu \in N_{i+1} \cup N_{i-1}} \chi_{[\tilde \nu]} = 0.\qedhere\]
\end{proof}

\subsection[The transport property and accuracy of the approximation]{The transport property and accuracy of the approximation\sectionmark{The transport property}}
\sectionmark{The transport property}

\begin{proof}[Proof of Lemma \ref{lemma:mass_in_approx_converges}]
The existence of $\tau_i\supdelta$ such that equation \eqref{delta_weight} and the upper bound in estimate \eqref{bounds_tau} hold is a direct consequence of the inequality \eqref{sigma_delta_leq_sigma} and the identity \eqref{disjoint_supports}.

\textit{Step 1: Rewrite the difference $ 18 \theta_i(1-\theta_i) - \tau_i\supdelta$ in terms of the partitions $\chi_\eta$ to exploit the bound on the interfacial energy.}\\
Let $\psi \in C_c^\infty(\Omega)$ and compute
\begin{align*}
  \int_\Omega  \tau_i\supdelta |\psi|^2 \intd \Leb^3 
  & \overset{\eqref{delta_weight}}{=} \sigma\supdelta_i(\psi \otimes 1) 
  \overset{\eqref{representation_gradient_h_measure}}{=} 2\mu_{i}(E_i,E_i;|\psi|^2\otimes 1)\\
  & \overset{\text{Def.}}{\underset{\phantom{\eqref{delta_weight}}}{=}} \lim_{\eta \to 0} 2\int_\Omega |\psi|^2\left(\partial_i \left( u_{i,\eta}\supdelta - u_i \right) \right)^2\intd \Leb^3
%  & = \lim_{\eta \to 0}  \int_\Omega |\psi|^2 \, 18 \left( \chi_{i,\eta} \supdelta - \theta_i\right)^2 \intd \Leb^3 = \lim_{\eta \to 0} \int_\Omega |\psi|^2\, 18\left(\left(\chi_{i,\eta}\supdelta\right)^2 - \theta_i^2\right) \intd \Leb^3.
\end{align*}
An application of the relations \eqref{relation_Du_chi} and \eqref{relation_Du_theta} gives
\[\lim_{\eta \to 0} 2\int_\Omega |\psi|^2\left(\partial_i \left( u_{i,\eta}\supdelta - u_i \right) \right)^2\intd \Leb^3  = \lim_{\eta \to 0} 2 \int_\Omega |\psi|^2  \left( 3 \chi_{i,\eta} \supdelta - 3 \theta_i + \chi_{0,\eta} - \theta_0\right)^2 \intd \Leb^3.\]
Note that the difference $\chi_{0,\eta} - \theta_0$ does not contribute in the limit due to
\[\lim_{\eta \to 0} \int_\Omega |\chi_{0,\eta} - \theta_0|^2 \intd \Leb^3 = \lim_{\eta \to 0} \int_\Omega \chi_{0,\eta} - 2\chi_{0,\eta}\theta_0 + \theta_0^2 \intd \Leb^3 = \int_\Omega \theta_0(1-\theta_0) \intd \Leb^3 = 0,\]
where in the last step we used $\theta_0 \in \{0,1\}$ almost everywhere, see Proposition \eqref{prop:disjoint_supports}.
Consequently, we get
\[\int_\Omega  \tau_i\supdelta |\psi|^2 \intd \Leb^3  = \lim_{\eta \to 0} \int_\Omega |\psi|^2\, 18\left(\left(\chi_{i,\eta}\supdelta\right)^2 - \theta_i^2\right) \intd \Leb^3.\]

The result of this computation can be used to deduce
\begin{align*}
 \int_\Omega |\psi|^2 \left(18\theta_i(1-\theta_i) - \tau_i\supdelta\right) \intd \Leb^3 & = \lim_{\eta \to 0} \int_\Omega |\psi|^2\, 18\left(\theta_i - \left(\chi_{i,\eta}\supdelta\right)^2\right) \intd \Leb^3 \\
 & = \lim_{\eta \to 0} \int_\Omega |\psi|^2\, 18\left(\chi_{i,\eta} - \left(\chi_{i,\eta} \supdelta\right)^2\right) 
 \intd \Leb^3 \\
 & = \lim_{\eta \to 0} \int_\Omega |\psi|^2\, 18\left(\chi_{i,\eta}^2 - \left(\varphi_{\delta\eta^\frac{1}{3}} \ast \chi_{i,\eta}\right)^2\right) \intd \Leb^3.
\end{align*}

\textit{Step 2: We have \[\liminf_{\eta \to 0} \int_\Omega|\chi_\eta -\varphi_{\delta\eta^{\frac{1}{3}}}\ast\chi_\eta| |\psi|^2\intd \Leb^3 \leq \delta E_{inter}(|\psi|^2).\]}\\
This is a $BV$-version of the well-known estimate
\[|| f - \phi_\delta \ast f ||_{L^p} \lesssim \frac{1}{\delta}||Df||_{L^p}\]
for $p \geq 1$.
We provide the argument to ensure that it also holds in the localized version we require.

For each $\eta$ in the subsequence let $\chi_\eta\supn$ be a smooth approximation of $\chi_\eta$ such that
\begin{enumerate}
 \item $\chi_\eta\supn \to \chi_\eta$ in $L^1(\Omega)$,
 \item $|D\chi_\eta \supn| \overset{\ast}{\warr} |D\chi_\eta|$
\end{enumerate}
as $n \to \infty$.
The existence follows from the usual density statement for $BV$ functions \cite[Theorem 2 of Chapter 5.2]{evans2015measure}, as convergence of the total mass and lower semi-continuity of the $BV$ norm on open subsets implies weak convergence of the total variation measures.
We estimate
\begin{align*}
& \quad \int_\Omega|\chi_\eta\supn -\varphi_{\delta\eta^{\frac{1}{3}}}\ast\chi_\eta\supn| |\psi|^2\intd \Leb^3\\
 & =  \int_\Omega\left|\int_{\ball{0}{\delta\eta^\frac{1}{3}}}\varphi_{\delta\eta^\frac{1}{3}}(y)\left(\chi_\eta\supn(x-y)-\chi_\eta\supn(x)\right)\intd y\right| |\psi|^2(x)\intd x\\
 & \leq \int_\Omega\int_{\ball{0}{\delta\eta^\frac{1}{3}}}\int_0^1\delta\eta^\frac{1}{3}\varphi_{\delta\eta^\frac{1}{3}}(y)|D\chi_\eta\supn|(x-ty)|\psi|^2(x)\intd t \intd y \intd x\\
 & =  \int_0^1\int_{\ball{0}{\delta\eta^\frac{1}{3}}} \int_\Omega \delta\eta^\frac{1}{3}\varphi_{\delta\eta^\frac{1}{3}}(y)|D\chi_\eta\supn|(x) |\psi|^2(x+ty)\intd x \intd y \intd t,
\end{align*}
where in the last step we used $\operatorname{supp} \psi \subset \subset \Omega$ and $\eta>0$ small enough when we shifted the domain of integration.
Letting $n$ go to infinity we obtain the estimate
\begin{equation*}
 \int_\Omega|\chi_\eta -\varphi_{\delta\eta^{\frac{1}{3}}}\ast\chi_\eta| |\psi|^2\intd \Leb^3 \leq \int_0^1\int_{\ball{0}{\delta\eta^\frac{1}{3}}} \int_\Omega \delta\eta^\frac{1}{3}\varphi_{\delta\eta^\frac{1}{3}}(y) |\psi|^2(x+ty)\intd |D\chi_\eta|(x) \intd y \intd t.
\end{equation*}
As a result of the convergence $|\psi|^2(x+ty) \to |\psi|^2(x)$ being uniform in $x$ and the measures $\eta^\frac{1}{3}|D\chi_\eta|$ having uniformly bounded mass we get
\begin{align}
 & \quad \liminf_{\eta \to 0} \int_\Omega|\chi_\eta -\varphi_{\delta\eta^{\frac{1}{3}}}\ast\chi_\eta| |\psi|^2\intd \Leb^3 \notag\\
 & \leq \lim_{\eta \to 0} \int_0^1\int_{\ball{0}{\delta\eta^\frac{1}{3}}} \int_\Omega \delta\eta^\frac{1}{3}\varphi_{\delta\eta^\frac{1}{3}}(y) |\psi|^2(x+ty)\intd |D\chi_\eta|(x) \intd y \intd t \notag\\
 & = \lim_{\eta \to 0} \int_\Omega \delta\eta^\frac{1}{3} |\psi|^2(x)\intd |D\chi_\eta|(x)\notag\\
 & = \delta E_{inter}(|\psi|^2). \notag
%  \label{convolutionconvergence2}
\end{align}

\textit{Step 3: Conclusion.}\\
Combining the results of Steps 1 and 2 we get
\[\int_\Omega |\psi|^2 \left(18\theta_i(1-\theta_1) - \tau_i\supdelta\right) \intd \Leb^3 \leq 36\, \delta E_{inter}(|\psi|^2).\]
Using $|\psi|^2$ we can approximate characteristic functions of balls $\overline{\ball{0}{r}} \subset \Omega$ to obtain
\[\int_{\overline{\ball{0}{r}}} \left(18\,\theta_i(1-\theta_1) - \tau_i\supdelta\right) \intd \Leb^3 \leq 36\, \delta E_{inter}(\overline{\ball{0}{r}}).\]
A differentiation theorem for Radon measures, see e.g.\ \cite[Theorem 1, Chapter 1.6]{evans2015measure}, implies
\[18\theta_i(1-\theta_1) - \tau_i\supdelta \leq 36\, \delta \frac{D  E_{inter}}{D\Leb^3}.\qedhere\]
%Approximating characteristic functions of balls 
%Because the measure on the left-hand side is absolutely continuous with respect to $\Leb^2$ this can be slightly improved to read
%\[\mu(e_1,e_1;\psi\otimes 1) - \mu\supdelta(e_1,e_1;\psi\otimes 1) \leq 4 \delta \int_\Omega \psi \frac{DE_{inter}}{D\Leb^2}\intd \Leb^2.\]
%Using $||a||_\infty - a$ instead of $a$ (and $\psi$ instead of $\psi^2$) in inequality \eqref{2D:monotonicity_nu_in_delta} we finally see
%\begin{align*}
% \nu(\psi \otimes a) - \nu\supdelta(\psi \otimes a) & \leq ||a||_\infty\left(\nu(\psi \otimes 1) - \nu\supdelta(\psi \otimes 1)\right)\\
%  & = ||a||_\infty\left(\nu(\psi \otimes 2\xi_1^2) - \nu\supdelta(\psi \otimes 2\xi_1^2)\right)\\
%  & = 2 ||a||_\infty\left(\mu(e_1,e_1;\psi \otimes 1) - \mu\supdelta(e_1,e_1;\psi \otimes 1)\right)\\
%  &\leq 8 ||a||_\infty \delta \int_\Omega \psi \frac{DE_{inter}}{D\Leb^2}\intd \Leb^2.\qedhere
%\end{align*}
%
%%The same calculation as in the 2D case can now be used to prove
%\[\int_\Omega \psi \left(\theta_i(1-\theta_1) - A_i\supdelta\right) \intd \Leb^3 \leq C\delta E_{inter}(\psi)\]
%for some $C>0$.
%As before this localizes to
%\[ \theta_i(1-\theta_i) - A_i\supdelta \leq C\delta \frac{DE_{inter}}{D\Leb^3}.\qedhere\]
\end{proof}

\begin{proof}[Proof of Proposition \ref{prop:transport_property}]
\textit{Step 1: Set up the notation and post-process Lemma \ref{lemma:H-measure_wave_equations}.}\\
Let $\delta>0$.
Let $i \in \{1,2,3\}$; $v,w \in \{[111],[\overline111],[1\overline11],[11\overline1]\}$ and $h_\eta\supdelta \in L^2(\Omega;\R^3)$ for each $\eta>0$ be such that
\begin{align}\label{transport_wave_equation}
	\partial_v\partial_w u_{i,\eta}\supdelta = \Div \varphi_{\delta \eta^\frac{1}{3}} \ast h_\eta^{(0)}
\end{align}
is one of the equations in Lemma \ref{lemma:H-measure_wave_equations}.
We will use the abbreviation 
\begin{align}
 U_{\eta}:=\left(u_{i,\eta}\supdelta-\varphi_{\delta\eta^{\frac{1}{3}}}\ast u_i\right)= \varphi_{\delta\eta^{\frac{1}{3}}}\ast\left(u_{i,\eta}- u_i\right).
%  \label{abbrev}
\end{align}
Because for all $\tilde v, \tilde w \in \R^3$ we have $\partial_{\tilde v}\varphi_{\delta\eta^{\frac{1}{3}}} \ast u_i \to \partial _{\tilde v} u_i$ strongly in $L^2$, the sequences $\partial_{\tilde v} U_\eta$, $\partial_{\tilde w} U_\eta$ of smooth functions still generate the H-measures $\mu_i\supdelta(\tilde v,\tilde w;\bullet \otimes \bullet)$.
Additionally, we drop the superscript of $h_\eta$.
The wave equation given above then reads
\[\partial_v\partial_w U_\eta = \Div \varphi_{\delta \eta^\frac{1}{3}} \ast h_\eta,\]
% \numberthis \label{transport_wave_equation}\]
which gives
\begin{align*}
	\int_\Omega |\psi |^2 |\partial_v\partial_w U_\eta|^2\intd \Leb^3 \leq \int_\Omega |\psi |^2\left|D\varphi_{\delta \eta^{\frac{1}{3}}}\ast h_{\eta}\right|^2\intd \Leb^3.
\end{align*}
Here the convolution on the right-hand side is understood to be componentwise.
As $\psi$ is uniformly continuous and $\varphi_{\delta\eta^\frac{1}{3}}$ concentrates in the limit $\eta \to 0$, we get
\begin{align*}
	\limsup_{\eta \to 0} \int_\Omega |\psi |^2 |\partial_v\partial_w U_\eta|^2\intd \Leb^3 \leq \limsup_{\eta \to 0} \int_\Omega \left|D \varphi_{\delta \eta^{\frac{1}{3}}}\ast (\psi h_{\eta})\right|^2\intd \Leb^3.
\end{align*}
Young's inequality, the scaling properties of $\left|\left|D\varphi_{\delta \eta^\frac{1}{3}}\right|\right|_{L^1}$ and the bound \eqref{3D:error_constraint} for $h_\eta$ imply
\begin{align}\label{transport_localized_estimate}
	\begin{split}
		\limsup_{\eta \to 0} \int_\Omega |\psi |^2 |\partial_v\partial_w U_\eta|^2\intd \Leb^3 & \lesssim \limsup_{\eta \to 0} \left|\left|D\varphi_{\delta \eta^{\frac{1}{3}}}\right|\right|^2_{L^1}\int_\Omega \psi^2 |h_{\eta}|^2 \intd x \\
   		& \leq  \frac{||D\varphi||_{L^1}^2}{\delta^2} E_{elast}(\psi^2). 
   	\end{split}
\end{align}

\textit{Step 2: Rewrite the distributional derivatives of $\mu_i\supdelta(v,v; \bullet \otimes \bullet)$ using the differential constraint \eqref{transport_wave_equation}.}\\
Let $\psi \in C_c^\infty(\Omega)$ and $a\in C(\Sph^1;[0,1])$.
As the derivatives of $U_\eta$ still generate the H-measures $\mu_i\supdelta$ we get
\begin{align*}
 \mu_i\supdelta(v,v;\partial_w |\psi|^2 \otimes a) & = \lim_{\eta \to 0} 2\RE \int a \F(\psi \partial_v U_\eta)\F^*(\partial_w \psi \partial_v U_\eta) \intd \Leb^3\\
 & = \lim_{\eta \to 0} 2\RE \int a \F(\psi \partial_v U_\eta)\F^*(\partial_w (\psi \partial_v U_\eta)) \intd \Leb^3 \\
 & \quad \quad -  2\RE \int a \F(\psi \partial_v U_\eta)\F^*(\psi \partial_w \partial_v U_\eta) \intd \Leb^3.
\end{align*}
The first term vanishes since
\[ \int a \F(\psi \partial_v U_\eta)\F^*(\partial_w (\psi \partial_v U_\eta)) \intd \Leb^3 = 2\pi i \int \xi\cdot w a |\F(\psi \partial_v U_\eta)|^2 \intd \Leb^3(\xi)\]
is purely imaginary.
Consequently, applying the Cauchy-Schwarz inequality to the second term and using inequality \eqref{transport_localized_estimate} to estimate the second derivatives we see that
\begin{align*}
	\left| \mu_i\supdelta(v,v;\partial_w |\psi|^2 \otimes a)\right|\lesssim \frac{1}{\delta} \left(\mu_i\supdelta(v,v;|\psi|^2 \otimes a) \right)^\frac{1}{2} \left(E_{elast}(\psi^2)\right)^\frac{1}{2}.
\end{align*}
%
%
%where in the last step we used \begin{align}\label{PIpseudodiff}
% \int m \F(Dv)\F^*(w) + m\F(v)\F^*(Dw) \intd \xi 
% & = \int i \xi m\F(v)\F^*(w) - i \xi m\F(v)\F^*(w) \intd \xi = 0
%\end{align}
%for $v,w \in W^{1,2}(\R^2;\R)$. Noticing that the right-hand side of the previous calculation is a sum of H-measures with support in $\{\xi_1^2 =\xi_2^2 = \frac{1}{2}\}$ we get
%\begin{align*}
% \mu(a,a;\partial_b \psi^2 \otimes m) & = -2 \mu(\partial_a v,\Div g_\eta; \psi^2 \otimes m) = - 4 \mu(\partial_a v,\Div g_\eta; \psi^2  \otimes \xi_1^2 m).
%\end{align*}
%Since the fact that second derivatives commute give us a differential contraint for the first derivatives, we can apply the localization principle once again to deduce
%\begin{align*}
% \mu(a,a;\partial_b \psi^2 \otimes m) = - 4 \mu\left(\partial_1 v,\Div g_\eta; \psi^2  \otimes \xi_1(\xi\cdot a) m\right).
%\end{align*}
%We now use the Cauchy-Schwarz inequality again to find the estimate
%\begin{align*}
% |\mu(a,a;\partial_b \psi^2 \otimes m)| & \leq 4 ||m||_\infty \lim_{\eta \to 0} \left|\left|\psi \partial_1v_\eta\right|\right|_{L^2} \limsup_{\eta \to 0} ||\psi \Div g_\eta||_{L^2} \\ 
%  & \leq 4 \sqrt{2} \frac{||D\varphi||_{L^1}}{\delta} 
%  \left(\int_\Omega \psi ^2 \left(1-\left(\partial_1u_1\right)^2\right)\intd \Leb^2\right)^\frac{1}{2} E_{elast,\eta}(\psi^2)^\frac{1}{2}. 
%\end{align*}

\textit{Step 3: Rewrite the result in terms of $\tau_i\supdelta \chi_{[\nu]}$.}\\
In terms of the measure $\sigma_i$ the last estimate reads
\begin{align*}
	 \left|\sigma_i\supdelta(\partial_w |\psi|^2 \otimes (\xi \cdot v )^2a) \right|\lesssim \frac{1}{\delta} \left(\sigma_i\supdelta(|\psi|^2 \otimes (\xi \cdot v )^2 a) \right)^\frac{1}{2} \left(E_{elast}(\psi^2)\right)^\frac{1}{2}.
\end{align*}
Using $a$ to localize around $\pm \nu$ for $\nu\in N_{i+1} \cup N_{i-1}$ with $\nu \cdot v \neq 0$ we get that
\begin{align*}
	\left|\int_\Omega \tau_i\supdelta \chi_{[\nu]} \partial_w |\psi|^2 \intd x \right| \lesssim \frac{1}{\delta} \left(\int_\Omega \tau_i\supdelta \chi_{[\nu]} |\psi|^2 \intd x\right)^\frac{1}{2} \left(E_{elast}(\psi^2)\right)^\frac{1}{2}.
\end{align*}

A straightforward crawl through the combinatorics in Lemma \ref{lemma:H-measure_wave_equations} reveals that for each $\nu\in N_{i+1} \cup N_{i-1}$ we have either $\nu \cdot v \neq 0$, $\nu\cdot w = 0$ or $\nu \cdot v = 0$, $\nu\cdot w \neq 0$.
Thus we see that each equation in Lemma \ref{lemma:H-measure_wave_equations} pertaining to $u_i$ for $i\in\{1,2,3\}$ allows us to estimate the weak derivative $\partial_w (\tau_i\supdelta \chi_{[\nu]})$ for one vector $w \in \{[111],[\overline111],[1\overline11],[11\overline1]\}$ with $w\cdot \nu =0$.
Furthermore, each of the two equations gives us an estimate for two linearly independent directions.
Consequently we can estimate $\partial_d (\tau_i\supdelta \chi_{[\nu]})$ for all directions lying in the two-dimensional subspace $\{\tilde d\cdot \nu = 0\}$ to get
\begin{align*}
	\left|\int_\Omega \tau_i\supdelta \chi_{[\nu]} \partial_d |\psi|^2 \intd x \right| \lesssim \frac{1}{\delta} \left(\int_\Omega \tau_i\supdelta \chi_{[\nu]} |\psi|^2 \intd x\right)^\frac{1}{2} \left(E_{elast}(\psi^2)\right)^\frac{1}{2}
\end{align*}
for $d\in \Sph^2$ with $d\cdot \nu=0$.

\textit{Step 4: Localize the estimate.}\\
As the right-hand side can be estimated by $||\psi^2||_\infty$ we see that $\partial_d \tau_i\supdelta \chi_{[\nu]}$ defines a finite Radon measure on $\Omega$.
Given any Borel set $B\subset \Omega$ we use $\psi^2$ to approximate its characteristic function and the value of all involved measures on it, leading to
\[\left| \partial_d \tau_i\supdelta \chi_{[\nu]} (B) \right| \leq \frac{C}{\delta} 
  \left(\int_B  \tau_i\supdelta \chi_{[\nu]}\intd x\right)^\frac{1}{2} E_{elast,\eta}(B)^\frac{1}{2}\]
  for some universal constant $C>0$.
 As the right-hand side vanishes for $\Leb^3$ null sets, we see that the derivatives are absolutely continuous with respect to $\Leb^3$.
 We then get the estimate \eqref{3D_def_meas_mixed} in all Lebesgue points.
\end{proof}

\begin{proof}[Proof of Lemma \ref{lemma:Besov_incomplete}]
	For any function $g: U \to \R$ let $g_h(x) := g(x + hd)$.
	We will use the abbreviations $\tau := 18 \theta_i(1-\theta_i)$, $\tau\supdelta := \tau_i\supdelta$ and $\chi:= \chi_{[\nu]}$, and remind the reader of the assumption $\tau > \eps$ almost everywhere on the set $\{\tau > 0\}$.
	Therefore, equation \eqref{support_sum_chi} in Proposition \ref{prop:disjoint_supports} implies $\tau(x) >\eps$ for almost all $x \in \Omega$ with $\chi(x) = 1$.
	Consequently, going through the cases $\chi_h(x) - \chi(x) \in \{-1,0,1\}$ we see for all $\delta >0 $ that
	\begin{align}\label{besov_partition}
	 \begin{split}
	  \int_U | \chi_h - \chi |  \intd x &  \leq \frac{1}{\eps} \int_U | (\tau\chi)_h - \tau\chi |\left((1- \chi) + (1-  \chi_h)\right)  \intd x \\
	  & \lesssim \frac{1}{\eps}\int_U | ((\tau - \tau\supdelta)\chi)_h|  + |(\tau\supdelta - \tau)\chi | \intd x \\
	  & \quad +  \frac{1}{\eps} \int_U  |(\tau\supdelta\chi)_h - \tau\supdelta\chi|  \left((1- \chi) + (1-  \chi_h)\right)  \intd x .
	 \end{split}
	\end{align}
	Applying the transport property, Proposition \ref{prop:transport_property}, and Lemma \ref{lemma:differential_inequality} to the third term we obtain
	\begin{align*}
	  & \quad \int_U  |(\tau\supdelta\chi)_h - \tau\supdelta\chi|  \left((1- \chi) + (1-  \chi_h)\right)  \intd x\\
	  & = \int_U  |(\tau\supdelta\chi)_h|  (1- \chi)  + | \tau\supdelta\chi| (1-  \chi_h)  \intd x\\
	  & \lesssim  \frac{h^2}{\delta^2} \int_U\dashint_0^h E_{elast}^\Leb(x+td) \intd t + \dashint_0^h E_{elast}^\Leb(x+h-td) \intd t \intd x.
	\end{align*}
	To get rid of the inner integrals we use Young's inequality for convolutions on the lines $t\mapsto x+td$.
	Additionally, we plug Lemma \ref{lemma:mass_in_approx_converges} into the first two terms on the right-hand side of the estimate \eqref{besov_partition} we get
	\begin{align*}
		\int_U | \chi_h - \chi | \intd x 
%		& \lesssim \frac{1}{\eps}\left(\delta E_{inter}^\Leb(U_h) + h^2\int_{U_h} |\partial_d \tau\chi|^2 \intd x\right) \\
	   \lesssim\frac{1}{\eps}\left(\delta E_{inter}^\Leb(U_h) + \frac{h^2}{\delta^2}E_{elast}^\Leb(U_h)\right) .
	\end{align*}
	Choosing $\delta := h^\frac{2}{3}\left(\frac{E_{elast}}{E_{inter}}\right)^\frac{1}{3}$ if $E_{inter}^\Leb(U_h), E_{elast}^\Leb(U_h) > 0$ and $\delta \to \infty$ or $\delta \to 0$ otherwise we see that
	\[\int_U | \chi_h - \chi | \intd x \lesssim \frac{1}{\eps}\left(E_{inter}^\Leb(U_h)\right)^\frac{2}{3}\left(E_{elast}^\Leb(U_h)\right)^\frac{1}{3} h^\frac{2}{3}.\qedhere\]
\end{proof}

\begin{proof}[Proof of Lemma \ref{lemma:differential_inequality}]
 For $t \in (0,1]$ we only have to deal with the case that $f(t) > 0$.
 Without loss of generality, we may assume that $ 0 = \inf \{ s\in (0,t): f(s)>0\}$.
 In that case we know $f^\frac{1}{2} \in W^{1,1}(0,t)$ with the pointwise a.e.\ estimate
 \[\left(f^\frac{1}{2}\right)'\leq g^\frac{1}{2}.\]
 The fundamental theorem of calculus for Sobolev functions implies that
 \[f^\frac{1}{2}(t) = f^\frac{1}{2}(t) - f^\frac{1}{2}(0) \leq \int_0^t g^\frac{1}{2}(s) \intd s.\]
 Squaring the inequality and applying Jensen's inequality to the right-hand side we get the desired statement
 \[f(t) \leq t^2 \dashint_0^t g(s) \intd s.\qedhere\]
\end{proof}

\subsection{Proofs of the main results}

\begin{proof}[Proof of Theorem \ref{thm:Besov_for_non-checkerboards}]
%	We will only prove the estimate for $\chi_{[0:1:1]} \in B^{2/3}_{1,\infty}$ as the proofs for the other statements work similarly.
	Let $\nu \in N_i$ for some $i=1,2,3$.
	We only have to prove the estimate for $\partial^h_\nu \chi_{[\nu]}$ since the difference quotients in directions $d$ with $d\cdot \nu =0$ are controlled by Lemma \ref{lemma:Besov_incomplete}.
	
	To this end let
	\[\mathcal{D} = \{[111],[\overline111],[1\overline11],[11\overline1]\}.\]
	Straightforward combinatorics imply that we have $d\cdot \nu \neq 0$ for exactly two directions $d_1,d_2 \in \mathcal{D}$ and that $d_1$ and $d_2$ uniquely determine $\nu \in N$.	
	Additionally, setting $\pi_i: \R^3 \to \R^2$ to be the projection dropping the $i$-th entry of a vector it can be seen that $\pi_i d_1= \pm \pi_i d_2 = \pm \pi_i \nu$.
	Possibly replacing $d_1$ by $-d_1$ or $d_2$ by $-d_2$ we may suppose that
	\begin{align}\label{label}
		\pi_i d_1=  \pi_i d_2 = \pi_i \nu.
	\end{align}
%	Furthermore, for a function $f:U \to \R$, $d \in \R^3$ and $h>0$ such that $h|d| \leq \dist (U,\partial \Omega)$ we write $\partial_d f \approx 0$ as an abbreviation for the statement
%	\[\sup_{h>0, h|d| \leq \dist (U,\partial \Omega)} h^{-\frac{2}{3}}\int_U | \partial_d^h f| \intd x < \infty,\]
%	The notation $\partial_d^h f$ was defined in Definition \ref{def:Besov}.
	
	By assumption there has to be some oscillation everywhere, i.e., we have
	\[\sum_{\nu \in N} \chi_{[\nu]} \equiv 1.\]
	Using that $\nu$ is uniquely determined by the property $d_1\cdot \nu \neq 0$ and $d_2 \cdot \nu \neq 0$ and applying Lemma \ref{lemma:Besov_incomplete} for all other normals this implies
	\[\int_U |\partial_{d_1}^h\partial_{d_2}^h\chi_{[\nu]} | \intd x \lesssim \frac{1}{\eps}\left(E_{inter}^\Leb(U_{ch})\right)^\frac{2}{3}\left(E_{elast}^\Leb(U_{ch})\right)^\frac{1}{3} h^\frac{2}{3}\]
	for $h> 0$ such that $h < \frac{1}{c} \dist(U,\partial\Omega)$.
	As $\partial_{E_i} \chi_{[\nu]}$ is controlled by Lemma \ref{lemma:Besov_incomplete} as well, we get using the normalizations \eqref{label} that
	\[\int_U |\partial_{\nu}^h\partial_{\nu}^h\chi_{\nu} | \intd x \lesssim \frac{1}{\eps}\left(E_{inter}^\Leb(U_{ch})\right)^\frac{2}{3}\left(E_{elast}^\Leb(U_{ch})\right)^\frac{1}{3} h^\frac{2}{3}.\]
	
	Thus to conclude we merely have to ensure that $|\partial_{d}^h\chi | \leq |\partial_{d}^h\partial_{d}^h\chi |$ for any measurable characteristic function $\chi$ and $d \in \R^3$.
	For almost all $x \in U$ we have $\partial_d^h \chi (x)\in \{-1,0,1\}$ and
	\[\partial_d^h \partial_{d}^h \chi(x) = \partial_d^h\chi(x+hd) - \partial_d^h\chi(x) = \chi(x+2hd) - 2\chi(x+hd) + \chi(x) \in \{-2,-1,0,1,2\}.\]
	In particular, we only have to ensure that $\partial_d^h \partial_{d}^h \chi(x) = 0$ implies $\partial_d^h \chi(x) = 0$.
	Indeed, if we have $\partial_d^h \partial_{d}^h \chi(x) = 0$ then straightforward combinatorics give
	\[\chi(x) = \chi(x+hd) = \chi(x + 2 hd).\qedhere\]
\end{proof}

\begin{proof}[Proof of Theorem \ref{thm:Besov_checkerboard}]
	By relabeling we may suppose $i=1$.
	Furthermore, we abbreviate
	\[E_h := \left(E_{inter}^\Leb(U_{ch})\right)^\frac{2}{3}\left(E_{elast}^\Leb(U_{ch})\right)^\frac{1}{3}\]
	for $h < \frac{1}{c} \dist(U,\partial \Omega)$ and $c\geq 1$ as in Theorem \ref{thm:Besov_for_non-checkerboards}.
	
	For $\nu \in N_1\cup N_2\setminus\{\nu_2\}$ there exist distinct $d_1, d_2 \in \mathcal{D}$ such that $d_1\cdot \nu_2=0$ and $d_1\cdot \nu, d_2 \cdot \nu \neq 0$, since we saw in the previous proof that the directions $d_1$ and $d_2$ uniquely determine $\nu$.
	Recall that by equation \eqref{support_sum_chi} there has to be some oscillation on the set 
	\[\{\theta_3>0\} = \{\theta_3=a\}.\]
	More specifically, we have $\sum_{\nu \in N_1 \cup N_2} \chi_{[\nu]} = \chi_{\{\theta_3 = a\}}$, which together with the fact that $\theta_3$ only depends on $x\cdot \nu_2$ implies
	\[\partial_{d_1} \sum_{\tilde \nu \in N_1 \cup N_2} \chi_{[\tilde \nu]} = 0.\]
	Taking a difference quotient in direction $d_2$ gives 
	\[\partial_{d_1} \partial_{d_2} \chi_{[\nu]} = 0.\]
	Consequently, we can use the same arguments as in the proof of Theorem \ref{thm:Besov_for_non-checkerboards} to see that
	\[\int_U |\partial_\nu^h \chi_{[\nu]}(x) |\intd x \leq  \frac{C}{\min(a,b)}E_h h^\frac{2}{3}\]
	for all $\nu \in N_1 \cup N_2 \setminus \{\nu_2\}$, $d\in \Sph^2$ and $h< \frac{1}{c}\dist(U,\partial\Omega)$.
	The same argument repeated for the set $\chi_{\{\chi_2 = b\}}$ tells us that
	\[\int_U |\partial_d^h \chi_{[\nu]}(x) |\intd x \leq  \frac{C}{\min(a,b)}E_h h^\frac{2}{3}\]
	for all $\nu \in N_1 \cup N_3 \setminus \{\nu_3\}$, $d\in \Sph^2$ and $h< c\dist(U,\partial\Omega)$.
	
	As $\operatorname{span}\left( \{d \in \R^3: d\cdot  \nu_3 = 0\} \cup \{\nu_2\}\right) = \R^3$ by $\nu_2 \cdot \nu_3 \neq 0$ due to definition \ref{def:normals}, $\nu_2\in N_2$ and $\nu_3 \in N_3$ we only have to prove
	\[\int_U |\partial_{\nu_2}^h \chi_{[\nu_3]}(x) |\intd x \leq  \frac{C}{\min(a,b)}E_h h^\frac{2}{3}\]
	in order to get the Besov-estimate in all directions $d \in \Sph^2$. 
	To this end, note that equation \eqref{support_sum_chi} for $i=2$ and $i=3$ together with the fact that there can locally only be a single direction of oscillation, see \eqref{disjoint_supports}, implies
	\[\sum_{\nu\in N_1} \chi_{[\nu]} = \chi_{\{\theta_1= 0,\, \theta_2 = b,\, \theta_3 = a\}} = \chi_{\tilde B }\chi_{\tilde A }\numberthis \label{decomposition_nu_1}\]
	for $\tilde A := \pi_{\nu_2}^{-1}(A)$ and $\tilde B := \pi_{\nu_3}^{-1}(B)$.
	Therefore, the proof so far gives the full Besov-estimate
	\[\int_U |\partial_d^h \chi_{\tilde B}\chi_{\tilde A } |\intd x \leq  \frac{C}{\min(a,b)}E_h h^\frac{2}{3}\]
	for all $d \in \Sph^2$ for the right-hand side.
	As $\chi_{\{\theta_2 = b\}}$ is independent of $\nu_2$ we obtain
	\[\int_U |\partial_{\nu_2}^h \chi_{\tilde B }(1-\chi_{\tilde A }) |\intd x \leq  \frac{C}{\min(a,b)}E_h h^\frac{2}{3}.\]
	Using the fact that we already proved the full estimate for $\chi_{[ \nu]}$ with $ \nu \in N_3 \setminus \{\nu_3\}$ and exploiting the equality
	\[\sum_{\nu\in N_3} \chi_{[\nu]} = \chi_{ \{\theta_1= 1-b,\, \theta_2 = b,\, \theta_3 = 0\}}=\chi_{\tilde B }(1-\chi_{\tilde A })\numberthis \label{decomposition_nu_2}\]
	we get the desired estimate for $\partial^h_{\nu_2} \chi_{[\nu_3]}$.
	
	The proof of the full estimate for $\chi_{[\nu_2]}$ works similarly.
	Thus we proved the Besov estimate for $\chi_{[\nu]}$ for all $\nu \in N$.
	
	Finally, we remark that the identity \eqref{decomposition_nu_1} also ensures that $\chi_{\{\theta_1= 0,\, \theta_2 = b,\, \theta_3 = a\}}$ satisfies the Besov estimate, while the estimate for $\chi_{\{\theta_1= 1-b,\, \theta_2 = b,\, \theta_3 = 0\}}$ is implied by $\eqref{decomposition_nu_2}$.
	Estimating the function $\chi_{\{\theta_1= 1-a,\, \theta_2 = 0,\, \theta_3 = a\}}$ works again similarly and to ensure that $\chi_{\{\theta_1= 1,\, \theta_2 = 0,\, \theta_3 = 0\}}$ is well-behaved we use
	\[\chi_{\{\theta_1= 1,\, \theta_2 = 0,\, \theta_3 = 0\}} \equiv 1 - \chi_{\{\theta_1= 0,\, \theta_2 = b,\, \theta_3 = a\}} - \chi_{ \{\theta_1= 1-b,\, \theta_2 = b,\, \theta_3 = 0\}} - \chi_{\{\theta_1= 1-a,\, \theta_2 = 0,\, \theta_3 = a\}}.\qedhere\] 
\end{proof}

\begin{proof}[Proof of Lemma \ref{energy_density_habit_plane}]
	For $d$, $h$ and almost all $x' \in \R^3$ as in the statement of the lemma and $\delta>0$ we can apply Proposition \ref{prop:transport_property} and Lemma \ref{lemma:differential_inequality} to get the upper bound
	\[\tau_2\supdelta (x'  + hd) = \tau_2\supdelta\chi_{[\nu_3]}(x' + hd)\lesssim \frac{h^2}{\delta^2}\dashint_0^h E_{elast}^\Leb(x' + s d) \intd s\]
	since $\chi_{[\nu_3]}(-\eps d) = 0$ for all $0<\eps <1$  by assumption.
	Lemma \ref{lemma:mass_in_approx_converges} gives us a corresponding lower bound
	\[18 \frac{2}{3}\left(1-\frac{2}{3}\right) - 36 \delta E_{inter}^\Leb(x' + hd) \leq \tau_2\supdelta (x' + hd).\]
	Combining both we see
	\[1 \lesssim \delta E_{inter}^\Leb(x' + hd) + \frac{h^2}{\delta^2}\dashint_0^h E_{elast}^\Leb(x' + s d) \intd s.\]
	Choosing
	\[\delta = h^\frac{2}{3}\left( \frac{\dashint_0^h E_{elast}^\Leb(x' + s d) \intd s }{E_{inter}^\Leb(x' + hd)}\right)^\frac{1}{3}\]
	 gives the statement
	\[ \left(E_{inter}^\Leb\right)^\frac{1}{3}(x' + h d) \left(\dashint_0^h E_{elast}^\Leb (x' + sd) \intd s \right) ^\frac{2}{3} \geq  C h^{-\frac{2}{3}}\]
	for a universal constant $C>0$.
\end{proof}

\subsection*{Acknowledgment}
The author thanks his PhD adviser Prof.\ Dr.\ Felix Otto for the many helpful discussions.
 
 \emergencystretch=1em
 
 \printbibliography
\end{document}